\documentclass[pagesize]{scrartcl}

\usepackage{mystyle}

\title{Deep BSDE Solver on Bounded Domains \\ Part I: General Loss Rate }
\author{
	Maximilian W\"{u}rschmidt\thanks{Trier University, Department IV -- Mathematics. M.\ W\"{u}rschmidt gratefully acknowledges financial support from the German Research Foundation (DFG) within the Research Training Group 2126: Algorithmic Optimization.} 
}
\date{}
\begin{document}
\allowdisplaybreaks
\maketitle
\begin{abstract}
    We consider a ramification of the deep BSDE loss functional designed to apply for BSDEs on bounded domains, i.e. with random (unbounded) time horizons.
    We derive a general convergence rate of the loss functional; precisely for a class of (randomly) weighted modifications of the functional. The rate is expressed in terms of the underlying discrete-time stepsize and a universal approximation distance.

		
\end{abstract}
\section{Introduction}

Neural network-based algorithms are increasingly relevant for the numerical solution of backward stochastic differential equations (BSDE) as well as for elliptic and parabolic partial differential equations (PDE), especially in high-dimensional problems. We refer to \cite{Beck2023, gonon2024overview, han2025brief, Richter2023} for an overview on neural PDE solvers, to \cite{chessari2023numerical} for a recent survey of numerical methods for BSDEs and to the monographs \cite{jentzen2023_BOOK, petersen2024_BOOK} for comprehensive introductions to deep learning theory.

We are concerned with a variant of a so-called \textit{deep BSDE solver} – a class relying on the Feynman-Kac link between BSDEs and PDEs – where the training consist of minimizing a loss functional that enforces consistency with a discrete-time forward scheme of the BSDE. The approach stems from the seminal works \cite{E2017, Han2018} with a loss functional purely penalizing deviations from the terminal condition. Shortly thereafter various ramifications of loss functionals were proposed; see e.g. 
\cite{andersson2025deep, beck2021deep,  chassagneux2023learning, pham2021neural, han2024deep,  JihunHan2020derivative, kremsner2020deep,  Pham2020deepBackward, kapllani2024deep, kremsner2020deep,  Raissi2018, zhang2022fbsde}. A fundamental question is whether those solvers are capable to circumvent the curse of dimensionality by achieving exponential expressivity utilizing networks with polynomial growth in the number of neurons, see e.g. \cite{beck2024overcoming, cioica2022deep, elbrachter2022dnn, grohs2023proofCoD, neufeld2025rectified}. Specifically for BSDEs with a bounded time horizon there are several theoretical advancements concerning convergence guarantees and rates; see e.g. \cite{frey2025convergence, gao2023convergence, gnoatto2025convergence, gonon2022uniform, grohs2023space, han2020convergence, Knochenhauer2021ConvergenceRates, negyesi2024generalized}.

%
\nopagebreak
To the best of our knowledge extensions of the deep BSDE method to the case of random unbounded time horizons – such as first exit times from bounded domains – is not yet thoroughly investigated. We briefly review existing literature. Frameworks for \textit{random unbounded exit times}: In \cite{furuya2024simultaneously} neural operators for BSDEs driven by a Brownian motion with constant multiplicative noise are investigated. The training uses a decomposition of the associated Green function exploiting a suitable wavelet expansion.
In \cite{matoussi2024_ErgodicDeep} possible extensions of the deep BSDE solver to ergodic BSDEs are presented in multiple numerical experiments. Another approach is based on the \textit{walk-on-spheres} method, see e.g. \cite{grohs2022deep, nam2024solving} where the methodology is applied to solve Poisson problems, particularly not involving the evaluation of any network derivative. 
%
Regarding the problem in terms of the \textit{elliptic PDE}, a similarly derivative-free approach for general linear elliptic PDEs with Dirichlet boundary condition is presented in \cite{JihunHan2024stochastic, JihunHan2020derivative} allowing a general collocation point distribution with support covering the entire domain. From the variational formulation of elliptic PDEs the \textit{deep Ritz} method was developed, see \cite{weinan2018deep} and for the linear case with Robin boundary there is error analysis, see e.g. \cite{jiao2024error, jiao2025drm}. Furthermore the expressivity of operator networks designed to apply for linear elliptic PDEs is investigated in \cite{marcati2023exponential}.

%
In this paper we analyse the asymptotic behaviour of a ramification of the deep BSDE solver. We consider a loss functional that learns the solution of the BSDE along an unbounded discrete-time grid, depending on the first exit of a simulation of the underlying forward diffusion. It is a natural extension of the loss functional proposed in \cite{Knochenhauer2021ConvergenceRates} which is based on the one introduced in \cite{Raissi2018}.  The key distinction from the classical deep BSDE approach lies in learning a single (neural) functional approximation of the solution of the BSDE for the entire discrete-time gird.

Our loss functional for random (unbounded) time horizons involves a random number of summands – in contrast to the constantly bounded time horizon with a finite sum of expectations. We analyse a \textit{randomly augmented} version of said loss functional; additionally a scaling factor given by a positive and bounded below random variable is introduced. For this augmented loss we establish a rate in terms of the underlying grid stepsize and a universal approximation distance (see Theorem~\ref{thm:general_LossRate}). The inclusion of the random scaling factor may facilitate future error analysis. As special case we present the rate for the plain loss (see Theorem~\ref{thm:special_LossRate}).

%
\textbf{Outline.} In Section~\ref{sec:setting} we introduce our loss functional and provide an informal description of our main result in terms of the plain loss functional, see Theorem~\ref{thm:special_LossRate}. In Section~\ref{sec:AssumptionsPreliminaryResultsEtc} we state the necessary regularity assumptions and collect useful results from literature. Our main results Theorem~\ref{thm:general_dynamical_LossRate} and ~\ref{thm:general_LossRate} are presented in Section~\ref{sec:GeneralLossRate}, with auxiliary results for the proof placed in Section~\ref{sec:AuxResults} and further useful bounds are provided in the supplementary material, see Appendix~\ref{sec:CombinedSupplements}.
\section{Random Time Horizon Loss Functional}\label{sec:setting}

Let $\domain\subset\R^d$ be a bounded domain and let $(\Omega, \F, \mathfrak{F}, \prob)$ be a filtered probability space, where the filtration $\mathfrak{F}$ is generated by a $d$-dimensional Brownian motion $W$ augmented by all $\prob$-nullsets. 
Under standard regularity assumptions (see Section~\ref{sec:AssumptionsPreliminaryResultsEtc}) we consider decoupled forward-backward SDEs on $\domain$. For $x\in\domain$ let $X^x$ denote the forward diffusion solving
\begin{equation}\label{eqnForwardDynamics}
    X_t^{x} = x + \int_0^t \mu(X_s^{x})\de s + \int_0^t \sigma(X_s^{x})\de W_s \qquad t\geq 0
\end{equation}
\pagebreak[0]
\nopagebreak[4]
where $ \tau^{x} \defined \inf \big\{ t\geq 0 \,\big|\, X_t^{x} \notin \domain \big\} $ denotes the first exit of $X^{x}$ from $\domain$. To keep the notation simple, we omit the starting point when there is no ambiguity and simply write $X \defined X^{x}$, $\tau\defined \tau^{x}$, etc. Let $(Y,Z)$ denote the unique solution of the BSDE
\begin{equation}\label{eqnRandomHorizonBSDE}
    Y_t =\; g(X_\tau) + \int_{t\wedge\tau}^{\tau} f(X_s,Y_s,Z_s) \de s - \int_{t\wedge\tau}^{\tau} Z_s \de W_s \qquad t \in [0,\tau].
\end{equation}
We use the well-known connection between a solution $(X,Y,Z)$ of the FBSDE \eqref{eqnForwardDynamics}-\eqref{eqnRandomHorizonBSDE} and the solution $u$ of the nonlinear second-order elliptic Dirichlet problem on $\domain$.\footnote{
    This is the well-known classical Feynman-Kac formula, see e.g.\ \cite[Section 5.7.2]{Pardoux2014}.
} 
Specifically let $\InfGen$ denote the generator of $X$, i.e. consider the second-order linear differential operator;
\begin{equation}\label{eqnInfGen}
	\InfGen[u](x) = \mu(x)^\top\! \Diff\! u(x) + \tfrac{1}{2}\trace\bigl[\sigma(x)\sigma(x)^\top \Diff^2\! u(x)\bigr]
\end{equation}
and let $u\colon\overline{\domain}\to\R$ denote the classical solution of the semilinear PDE;
\begin{equation}\label{eqnPDE}
	\InfGen[u] + f (\cdot, u, \sigma^\top \Diff\! u) = 0 \quad\text{on }\domain 
	\qquad\text{and}\qquad 
	u=g\quad\text{on } \partial\domain
\end{equation}
Then $u$ and the solution $(Y,Z)$ of \eqref{eqnRandomHorizonBSDE} are related via the classical Feynman-Kac correspondence
\begin{equation}\label{eqnIdentification}
	Y_t = u(X_t)
	\qquad\text{and}\qquad 
	Z_t = \sigma(X_t)^\top\Diff u(X_t)
	\qquad \text{on } \big\{t\in[0,\tau]\big\} 
\end{equation}
For a given stepsize $h\in(0,1)$ let $ \unboundedTimeGrid \defined \{ \ell h \ |\ \ell\in\nat_0 \} $ denote equidistant discrete-time grid.
Fix an admissible class $\Fct$ of functional approximations.\footnote{
    We assume that $\Fct$ is a bounded subset of the Hölder space $\HoelderSpace{3}{\gamma} (\R^d)$, see Definition~\ref{Definition Hoelder Space}. E.g. fully connected single hidden-layer neural networks with activation function of class $\HoelderSpace{3}{\gamma}(\R^d)$ are a possible choice for $\Fct$.
} 
We consider the loss functional;
\begin{equation}\label{defLossFunctional}
	\mathfrak{L} \colon\Fct\rightarrow\R;\quad \nn \mapsto \E \Big[ \lossBdry^{\StoppD} (\nn) + \loss^{\StoppD} (\nn) \Big]
\end{equation}
where $\lossBdry^{\StoppD} (\nn) \defined |\nn(\X_{\StoppD}) - g(\X_{\StoppD})|^2$ is a penalization of deviations from the terminal condition of \eqref{eqnRandomHorizonBSDE} and $\loss^{\StoppD} (U)$ is a dynamical loss random variable specified in \eqref{defLossRandomVariable} below.\footnote{
	Informally this is a penalization of a \textit{forward discretization} of the BSDE \eqref{eqnRandomHorizonBSDE} along trajectories of \eqref{eqnForwardDynamics}; carried out by the classical Euler-Maruyama scheme $\X$ and a discrete-time approximation $\StoppD$ of the first exit $\tau$.
}
Our loss formulation differs from the one considered in deep BSDE solvers with constant time horizon in that we consider a \emph{random} sum of penalization terms determined by the number of Euler-Maruyama jumps until the first exit $\StoppD$. To be precise, the loss random variable is
\begin{equation}\label{defLossRandomVariable}
	\loss^{\StoppD} (\nn) \defined 
    \sum_{ t\in\unboundedTimeGrid\cap [0,\StoppD) } \big|\lossSummand{t}{\unboundedTimeGrid} (\nn) \big|^2
\end{equation}
where for each $t\in\unboundedTimeGrid$ the dynamical penalization $\lossSummand{t}{\unboundedTimeGrid} (\nn)$ is defined as follows
\begin{equation}\label{defLossSummand}
	\begin{aligned}
		\lossSummand{t}{\unboundedTimeGrid}(\nn) &\defined \nn(\X_{t+h}) - \nn(\X_{t}) \\
		&\hspace{1.0cm} + f\big(\X_{t}, \nn(\X_{t}),[\sigma^\top\Diff\!\nn](\X_{t})\big) h \\
		&\hspace{2.0cm} - [\Diff\!\nn^\top\sigma](\X_{t})(W_{t+h}-W_{t}) \\
		&\hspace{3.0cm} - [\mu^\top\Diff^2\!\nn\sigma](\X_{t}) h (W_{t+h}-W_{t}) \\	
		&\hspace{4.0cm} - \tfrac{1}{2} (W_{t+h}-W_{t})^\top [\sigma^\top\Diff^2\!\nn \sigma] (\X_{t}) (W_{t+h}-W_{t}) \\
		&\hspace{5.0cm} + \tfrac{1}{2} \trace[\sigma^\top\Diff^2\!\nn \sigma](\X_{t}) h
	\end{aligned}
\end{equation}
where $\X$ denotes the well-known classical Euler-Maruyama scheme; set $\X_0 \defined x$ and define
\begin{equation}\label{defEulerMaruyama}
	\X_{t+h} \defined \X_{t} + \mu (\X_{t}) h + \sigma (\X_{t}) ( W_{t+h}-W_{t} ) \qquad t\in\unboundedTimeGrid
\end{equation}
Correspondingly we consider
$ \StoppD \defined \inf \{t\in\unboundedTimeGrid \ |\ \X_t \notin \domain \} $
as discrete-time approximation of $\tau$.

The goal of this paper is to establish a convergence rate of the expected loss functional $\mathfrak{L}$ defined in \eqref{defLossFunctional}. We analyze a generalized version of $\mathfrak{L}$, where additionally a positive random variable $\Psi$ is introduced. This can be regarded as random augmentation acting as multiplicative weight on the loss random variable $\loss^{\StoppD} (\nn)$ (see \eqref{defLossFunctional}). We consider 
\begin{equation}\label{defLossFunctional_PSI}
	\mathfrak{L}^\Psi \colon\Fct\rightarrow\R;\quad \nn \mapsto \E \Big[ \Psi \big( \lossBdry^{\StoppD} (\nn) +   \loss^{\StoppD} (\nn) \big) \Big]
\end{equation}
We establish a convergence rate of $\mathfrak{L}^{\Psi}$ consisting of two summands: a rate in terms of the stepsize $h$ and a distance that can be minimized using the universal approximation capabilities of neural networks. 
For mappings $\psi\in\C^{2} (\overline{\domain})$ we consider the norm\footnote{
    Notice this norm corresponds to the \textit{uniform on compacta universal approximation} result from \cite{Hornik1991}
}
\begin{equation}\label{defUniversalApproxNorm}
    {\| \psi \|}_\domain \defined \sup_{x\in \overline{\domain}}\Bigl[\bigl| \psi(x) \bigr| + \bigl\|\Diff\!\psi(x) \bigr\| + \bigl\|\Diff^2\!\psi(x) \bigr\|\Bigr].
\end{equation}
For illustrative purpose we state a special case of our main result Theorem~\ref{thm:general_LossRate}: The rate of the plain unscaled loss functional $\mathfrak{L}$ (cf. \eqref{defLossFunctional});
\begin{theorem}\label{thm:special_LossRate}
    Assume the setting of Theorem~\ref{thm:general_LossRate}. There are $C_1, C_2 >0$ (not depending on $h$);
    \begin{equation}
        \E \big[ \lossBdry^{\StoppD} (\nn) \big] \leq C_1 h^\frac{1}{4}
        \qquad\text{and}\qquad
        \E \big[ \loss^{\StoppD} (\nn) \big] \leq C_2 h^{1+\frac{1}{4}} + C_2 {\|U-u\|}_\domain^2 . \closeEqn
    \end{equation}
\end{theorem}
\begin{proof}
    This is a direct consequence of Theorem~\ref{thm:general_LossRate} for the choice $\Psi = 1$. Notice the rate of the boundary penalization improves since the Cauchy-Schwarz inequality is not required in the first step of the proof.
\end{proof}
\section{Assumptions and Results from Literature}\label{sec:AssumptionsPreliminaryResultsEtc}

Let $\theta\in\R$. Following the seminal works \cite{PardouxRandomTimeBSDE, Peng1991} on BSDEs with random time horizon we are concerned with solutions of \eqref{eqnRandomHorizonBSDE} satisfying $(Y,Z)\in\M_{\theta}^2(\tau; \R)\times\M_{\theta}^2(\tau; \R^d)$.\footnote{
    We write 
    $\M_{\theta}^2(\Stopp; \R^n) \defined \{ \Process\ \text{progressive,}\ \R^n\text{-valued}\ |\ \E [ \int_0^{\Stopp} \exp (\theta t)\ |\Process_t|^2 \de t ] < \infty \} $
    for a stopping time $\Stopp$.
}
For further details we refer to \cite[Section 5.7.2]{Pardoux2014}, \cite[Section 4.6]{zhang2017backward} as well as \cite{briand2003lp, BriandHuHomogenization, lin2020second, papapantoleon2018existence, royer2004bsdes} and the references contained therein. 

For $k\in\nat_0$ and $\gamma\in(0, 1)$ we say that a function is of class $\HoelderSpace{k}{\gamma}$ if it is $k$-times continuously differentiable with $\gamma$-Hölder continuous derivatives, and we denote the space of functions of class $\HoelderSpace{k}{\gamma}$ on $\domain$ by $\HoelderSpace{k}{\gamma} (\domain)$.
For further details we refer to \cite[p.52]{GilbargTrudinger}. 
\begin{domain_assumption}\label{assumption:domain}
    %
    %
    The domain $\domain$ is bounded and its boundary $\partial\domain$ is of class $\HoelderSpace{3}{\gamma}$ for some $\gamma\in(0,1)$, i.e.\ $\partial\domain$ admits a representation via maps of class $\HoelderSpace{3}{\gamma}$. \close
\end{domain_assumption}
\begin{fbsde_assumption}\label{assumption:fbsde} 
    %
    %
    $\mu\colon\R^d\rightarrow\R^d$ is of class $\HoelderSpace{1}{\gamma} (\overline{\domain})$, $\sigma\colon\R^d\rightarrow\R^{d\times d}$ is of class $\C^2$
    and there is $R>0$ such that $\sigma\sigma^\top$ is uniformly elliptic on $\domain+B_{R}$ where $B_R$ denotes the open ball centered at zero with radius $R$.\footnote{
    	I.e., $\sigma\sigma^\top (x)$ is positive definite for all $x\in\domain+B_{R}$ with uniformly bounded eigenvalues (particularly uniformly bounded away from zero); see \cite[p.31]{GilbargTrudinger}.
    }
    Moreover $f\colon\R^d\times\R\times\R^d\rightarrow\R$ is of class $\C^{1}$, there is $C \geq 0$ such that
    \begin{align*}
        \sign (u) f(x,u,p) \leq C (\|p\| + 1), \quad (x,u,p)\in\domain\times\R\times\R^d 
    \end{align*}
    and the structure conditions are satisfied, i.e.\ for every $M>0$ there exists $C(M) \geq 0$ such that
    \begin{align*}
        \text{I)}\qquad & \lim_{\|p\|\to\infty}\; \sup_{x\in\domain, |u|\leq M}\; \tfrac{1}{\|p\|^2} \Big( \big| f(x,u,p) \big| + \big|\big\langle p, \Diff_p f(x,u,p) \big\rangle\big| \Big) \leq C(M) \\
        \text{II)}\qquad & \lim_{\|p\|\to\infty} \; \sup_{x\in\domain, |u|\leq M}\; \tfrac{1}{\|p\|^2} \Big| \Diff_u f(x,u,p) + \tfrac{1}{\|p\|^2} \big\langle p, \Diff_x f(x,u,p) \big\rangle \Big| = 0.
    \end{align*}
    Finally $g\colon\R^d\rightarrow\R^d$ is of class $\HoelderSpace{3}{\gamma}(\overline{\domain})$ for some $\gamma\in (0,1)$. \close
\end{fbsde_assumption}
Under these assumptions for every $x\in\domain$ the SDE \eqref{eqnForwardDynamics} has a unique solution that is contained in $\mathcal{S}^p(\tau)$ for any $p>1$.\footnote{
    Existence and uniqueness holds under weaker assumptions, see e.g. 
    \cite[Theorem 1.2]{Krylov1999}. Particularly the solution is bounded on $[0,\tau]$ and thus indeed contained in $\mathcal{S}^p (\tau) \defined \{ \Process \ |\ \E[ \sup_{t\in [0,\tau]} \| \Process_t \|^p ] < \infty \}$ for every $p>1$.
} 
Additionally we assume that $\tau$ has a sufficiently large positive exponential moment. This enables us to obtain uniqueness of a solution of the BSDE \eqref{eqnRandomHorizonBSDE} resp. of the PDE \eqref{eqnPDE} (see Theorem~\ref{thm:ExistenceAndUniqueness}) from the results in \cite{PardouxRandomTimeBSDE}.\footnote{
    Assumption \hyperref[assumption:exp_bsde]{(e-$1$)} may be regarded as analogous to \cite[(25) resp. (72)]{PardouxRandomTimeBSDE} where we assume that $f$ is Lipschitz in $y$.
}
\begin{exp_bsde_assumption}\label{assumption:exp_bsde}
%
%
    Assume $f$ is Lipschitz continuous and let $\lipschitz{f}(y), \lipschitz{f}(z) \geq 0$ be the constants (w.r.t. $y$ resp. $z$).\footnote{
    	Let $\mathcal{O}\subseteq\R^n$ and $\psi\colon\mathcal{O}\to \R^m$ be Lipschitz continuous.
		Canonically we let $\lipschitz{\psi}\geq 0$ denote the Lipschitz constant.
    }
    Assume there is $\rho > \lipschitz{f}(z)^2 + 2 \lipschitz{f}(z)$ such that
    $\sup_{x\in\domain} \E [ \exp(\rho \tau^x) ] < \infty$. \close
\end{exp_bsde_assumption}
\begin{theorem}\label{thm:ExistenceAndUniqueness}
    Assume \hyperref[assumption:domain]{\emph{(Dom)}} and \hyperref[assumption:fbsde]{\emph{(FBSDE)}} hold for the same $\gamma\in (0,1)$ and there is $\rho>0$ such that \hyperref[assumption:exp_EM]{(e-$2$)} is satisfied.
    Then the BSDE~\eqref{eqnRandomHorizonBSDE} has a unique solution $(Y,Z) \in \M_{\rho}^2(\tau; \R)\times\M_{\rho}^2(\tau; \R^d)$ and the PDE~\eqref{eqnPDE} has a unique solution $u\in \HoelderSpace{3}{\gamma} (\overline{\domain})$. \close
\end{theorem}
\begin{proof}
    By \cite[Theorem 15.10]{GilbargTrudinger} there exists a solution $u$ of \eqref{eqnPDE} of class $\HoelderSpace{2}{\gamma}(\overline{\domain})$. Let $x\in\domain$ and suppose that $\bar u\in\HoelderSpace{2}{\gamma}(\overline{\domain})$ is another solution of \eqref{eqnPDE}. Set $\widetilde{Y}_t\defined \bar{u}(X_t)$ and $\widetilde{Z}_t\defined \sigma^\top(X_t)\Diff\! \bar{u}(X_t)$ then $(\widetilde{Y}, \widetilde{Z})\in \M_{\rho}^2(\tau; \R)\times\M_{\rho}^2(\tau; \R^d)$ is a solution of \eqref{eqnRandomHorizonBSDE}. 
    The same holds for $(Y,Z)$ if we set $Y_t\defined u(X_t)$ and $Z_t\defined \sigma^\top(X_t)\Diff\! u(X_t)$. From \cite[Theorem 3.4]{PardouxRandomTimeBSDE} we obtain uniqueness of the BSDE solution, i.e. $Y=\widetilde{Y}$ and in particular $u(x)=Y_0=\widetilde{Y}_0=\bar u(x)$. Finally setting $\widetilde{u} \defined u$ and $\widetilde{f} \defined f (\cdot, \widetilde{u}, \sigma^\top \Diff\! \widetilde{u})$ it follows that $u$ is a solution of the linear Dirichlet Problem
    $ \InfGen [u] + \widetilde{f} = 0\ \text{on}\ \domain$ and $u=g\ \text{on}\ \partial\domain $ whence \cite[Theorem 6.19]{GilbargTrudinger} implies $u\in\HoelderSpace{3}{\gamma} (\overline{\domain})$.
\end{proof}
In the following result (see Theorem~\ref{thm:PreliminaryResultsBGG}) we collect results from \cite{BGG_forward_exit} which will be used in our analysis below.
In order to apply them we need the following assumption.
\begin{exp_EM_assumption}\label{assumption:exp_EM}
%
%
    Assume that there are $\alpha >0$, $q\geq 4$ and $\beta > \tfrac{qd}{q-1} (6\lipschitz{\mu} +3q\lipschitz{\sigma}^2)$ such that $\prob [ \tau \geq k ] \leq \alpha \exp (-\beta k)$ for all $k\in\nat_0$. \close
\end{exp_EM_assumption}
\begin{theorem}\label{thm:PreliminaryResultsBGG}
    Let $\mu$ and $\sigma$ be Lipschitz and bounded. Assume there is $R>0$ such that $\sigma\sigma^\top$ is uniformly elliptic on $\domain + B_R$, where $B_R$ denotes the open ball centered at zero with radius $R$.
    If $\partial\domain$ is of class $\C^2$ then there is a constant $c>0$ such that 
    $\E [ \exp (c\StoppD) ] < \infty$.
    
    Particularly if \hyperref[assumption:exp_EM]{(e-$2$)} is satisfied, then for each $p\in\nat$ there is $c_p>0$ such that for all $h\leq h_0$;
    \begin{align}\label{eqn:EulerSchemeExitTime_and_Space_Error}
        \E \big[ | \tau - \StoppD |^p \big] \leq c_p h^{\frac{1}{2}} 
        \qquad\text{and}\qquad
        \E\Big[ |X_{\tau} - \X_{\StoppD}|^2 \Big] \leq\, C h^{\frac{1}{2}}. \closeEqn
    \end{align}
\end{theorem}
\begin{proof}
    By assumption $\sigma\sigma^\top$ is uniformly elliptic on $\domain_R \defined \domain + B_R$. Let $\nu >0$ denote the respective lower bound, i.e. $\xi^\top[\sigma\sigma^\top](x)\xi \geq \nu \xi^\top\xi$ for all $x\in\domain_R$, $\xi\in\R^d$. Then particularly
    $ ([\sigma\sigma^\top] (x))_{ii} \geq \nu $ for all $x\in\domain_R$ and each $i\in\{1,\dots,d\}$. For $x\in\domain_R$ let $\tau_{R}^x$ denote the first exit of $X^x$ (see \eqref{eqnForwardDynamics}) from $\domain_R$. From \cite[Lemma III.3.1]{Freidlin1985} we obtain $c_1 > 0$ such that
    $ \sup_{x\in\domain_R} \E [ \tau_{R}^x ] \leq c_1 < \infty $. Using the Markov inequality we conclude 
    \begin{equation}
        \sup_{x\in\domain_R} \prob [ \tau_{R}^{x} \geq T ] \leq \sup_{x\in\domain_R} \tfrac{1}{T} \E [ \tau_{R}^x ] \leq \tfrac{c_1}{T} \quad\text{for all } T>0
    \end{equation}
    With this we can apply \cite[Lemma A.4]{BGG_forward_exit} and obtain a constant $L>0$ such that
    \begin{equation}\label{eqn:BGG_Assumption_L}
        \sup_{x\in\domain} \E_{\Stopp} \big[ \Bar{\theta}^{x} (\Stopp) - \Stopp \big] + \sup_{x\in\domain} \E_{\Stopp} \big[ \theta^{x} (\Stopp) - \Stopp \big] \leq L \quad\text{for all }\Stopp\ \text{finite stopping time}
    \end{equation}
    where $\Bar{\theta}^{x} (\Stopp) \defined \inf \{ t\geq\Stopp \ |\ t\in\unboundedTimeGrid, X_t^x\notin\domain \}$ and $\theta^x (\Stopp) \defined \inf \{ t\geq\Stopp \ |\ X_t^x \notin \domain \}$. Notice \eqref{eqn:BGG_Assumption_L} is \cite[Assumption (L)]{BGG_forward_exit}, whence with \cite[Lemma 2.8]{BGG_forward_exit} we obtain $c>0$ such that $\exp(c\StoppD) \in L^1$. 

    The second part is immediate from \cite[Theorem 3.4]{BGG_forward_exit}. Regarding the prerequisites observe that by assumption $\partial\domain$ is of class $\C^2$, thus we can construct the required signed distance function, see e.g. \cite[Section 14.6]{GilbargTrudinger}, and moreover \hyperref[assumption:exp_EM]{(e-$2$)} ensures the necessary tail behaviour of $\tau$. 
\end{proof}
Next we state a corollary of \cite[Proposition 10.4]{Seifried2025bsdeDiscretization}.\footnote{
    With the notation of the reference; we mention that the corollary below is obtained by applying \cite[Proposition 10.4]{Seifried2025bsdeDiscretization} result for $q=4$ and $\varepsilon=\tfrac{1}{4}$.
} 
For this we require an assumption on the tail decay, which in contrast to Assumption~\hyperref[assumption:exp_EM]{(e-$2$)} is concerned with both $\tau$ and $\StoppD$.
\begin{exp_EM_bound_assumption}\label{assumption:exp_EM_bound}
    %
    %
    Assume there is a constant $\alpha >0$ and a scalar $\beta > 8d (6\lipschitz{\mu} + 24 \lipschitz{\sigma}^2)$ such that $\prob [\tau\vee\StoppD \geq k] \leq \alpha\exp(-\beta k)$ for all $k\in\nat_0$. \close
\end{exp_EM_bound_assumption}
\begin{corollary}\label{corollary:BSDE_Discretization_EM_Stopp}
    Let $\mu$ and $\sigma$ be Lipschitz and bounded. Assume Assumption~\hyperref[assumption:exp_EM_bound]{(e-$3$)} is satisfied, then there is $C>0$ such that for all $h\in(0,1)$ it holds that
    \begin{equation}
        \E \big[ \| X_{\tau} - \X_{\StoppD} \|^4 \big] \leq C h^2 + C \E \big[ | \tau-\StoppD|^{6} \big]^\frac{1}{2} + C \E \big[ | \tau-\StoppD|^{2} \big]^\frac{1}{2} . \closeEqn
    \end{equation}
\end{corollary}
We conclude this section by stating properties of the Euler-Maruyama scheme \eqref{defEulerMaruyama}, and of a typically introduced adapted and continuous interpolation, see e.g. \cite{KloedenPlaten}. We consider 
\begin{equation}\label{def:EulerMaruyamaContinuousInterpolation}
	\Xc_t \defined x + \int_0^t  \sum_{\ell\in\unboundedTimeGrid} \1_{ \{s\in(\ell, \ell+h]\} } \mu (\X_{\ell}) \de s + \int_0^t \sum_{\ell\in\unboundedTimeGrid} \1_{ \{s\in(\ell, \ell+h]\} } \sigma (\X_{\ell}) \de W_s\quad t\geq 0
\end{equation}
and let $\XcExit$ denote the first exit of $\Xc$ from $\domain$.\footnote{
    Notice $\XcExit \leq \StoppD$. Thus whenever there is $c>0$ such that $\exp(c\StoppD) \in L^1$ then particularly $\exp(c\XcExit)\in L^1$. See Theorem~\ref{thm:PreliminaryResultsBGG} for sufficient conditions.
}
Then it is well-known that;
\begin{proposition}[Properties of stopped Euler-Maruyama]\label{prop:PropertiesOfStoppedEuler}
    Let $q\geq 1$ and $t\in\unboundedTimeGrid$. Then\footnote{
        For a given map $\psi\colon\R^d\to \R$ we use the notation $\supremum{\psi} \defined \sup_{x\in \overline{\domain}} |\psi (x)| $. 
    }
   \begin{equation}
        \|\Xc_{s} - \Xc_{t}\|^q \leq 2^{q-1} \big( \supremum{\mu}^q\, h^q + \supremum{\sigma}^q\, \| W_{s} - W_{t} \|^q \big) \quad\text{for all } s\in[t,t+h] \quad\text{on } \{t <\StoppD \}. \closeEqn
    \end{equation}
\end{proposition}
\begin{proof}
    Let $t\in \unboundedTimeGrid$ and $s\in [t, t+h]$. By definition of the interpolation we have 
    \begin{equation}
        \Xc_{s} - \Xc_{t} = \int_{t}^{s} \mu (\X_t) \de r + \int_{t}^{s} \sigma (\X_t) \de W_r = \mu (\X_t) (s-t) + \sigma (\X_t) (W_s-W_t)
    \end{equation}
    On $ \{ t < \StoppD \}$ we have $\X_t \in \domain$ and thus we obtain the claim from an elementary bound.\footnote{
        It holds that $|a+b|^q \leq 2^{q-1} |a|^q + 2^{q-1} |b|^q$ for any $a,b\in\R$
    }
\end{proof}
\section{General Loss Rate}\label{sec:GeneralLossRate}

In this section we present our main results: Theorem~\ref{thm:general_dynamical_LossRate} is a bound of the dynamical penalization term of the loss functional (cf. \eqref{defLossFunctional_PSI}) and as a consequence we derive Theorem~\ref{thm:general_LossRate} -- the rate in terms of the underlying stepsize $h$ and a universal approximation distance between the candidate $\nn$ and the exact solution $u$ of \eqref{eqnPDE}. 

The restriction $h\leq h_0$ is due to the use of Theorem~\ref{thm:PreliminaryResultsBGG} which collects results from \cite{BGG_forward_exit}.
\begin{theorem}[Dynamical Loss Rate]\label{thm:general_dynamical_LossRate}
    Let $\Psi$ denote a random variable that is bounded away from zero and satisfies $\Psi \in L^{p}$ for some $p\geq 32$. 
    Assume \hyperref[assumption:domain]{\emph{(Dom)}} and \hyperref[assumption:fbsde]{(FBSDE)} hold for the same $\gamma\in(0,1]$ and assume \hyperref[assumption:exp_bsde]{(e-$1$)} and \hyperref[assumption:exp_EM]{(e-$2$)} are satisfied.
    Then there is $C >0$ (not depending on $h$);
    \begin{align*}
        \E \big[\Psi \ \loss^{\bar\tau} (U) \big] \leq C h^\frac{3}{2} + C h \E \big[ (\StoppD - \XcExitPlus)^2 \big]^\frac{1}{2} + C {\|U-u\|}_\domain^2 \quad\text{for all } h\leq h_0
    \end{align*}
    where $\XcExitPlus \defined \min \{ t\in\unboundedTimeGrid \,|\, \XcExit \leq t \}$ is the first discrete-time after $\XcExit$. \close
\end{theorem}
Before presenting the proof we briefly sketch the overall strategy. Building on the fact that the exact solution $u$ of \eqref{eqnPDE} only satisfies the PDE on $\domain$, we consider an intrinsic decomposition of the loss random variable $\loss^{\StoppD}$ (cf. \eqref{defLossRandomVariable}).\footnote{
	Under our standing assumptions it is possible to extend $u\in\HoelderSpace{3}{\gamma} (\overline{\domain})$ to $\R^d$ while preserving the subsequently required regularity, see e.g. \cite[Lemma 2.20]{LeeSmoothManifolds}. 
	However for the extension we merely achieve a rate of magnitude as the intrinsic loss rate admits, since we cannot use the error estimate for the distance between $\StoppD$ and $\XcExit$.
}
We separate $\loss^{\StoppD}$ at the first exit $\XcExit$ of the Euler-Maruyama interpolation, i.e. we split the sum into one part that excludes excursions of $\Xc$ and another that has no (suitable) access to the exact solution of the PDE. The details of the decomposition are presented in Proposition~\ref{prop:Identification_LossSummand_with_PrePost}; the proof of Theorem~\ref{thm:general_dynamical_LossRate} consists of the individual bounds. 
We require additional notation. For $0\leq s_1\leq s_2 < \infty$ and $x_1,x_2 \in \R^d$ we set
\begin{equation}\label{defLossSummand_GeneralVersion}
    \begin{aligned}
        \mathcal{L}_{s_1,s_2} (\nn) (x_1,x_2) &\defined \nn(x_2) - \nn (x_1)    \\
        &\hspace{0.5cm} + f\big(x_1, \nn(x_1), [\sigma^\top\!\Diff\!\nn ](x_1) \big) (s_2-s_1)    \\
        &\hspace{1.0cm} - [\Diff\!\nn^\top\sigma] (x_1) (W_{s_2}-W_{s_1})     \\
        &\hspace{1.5cm} - [\mu^\top\!\Diff^2\!\nn\sigma](x_1) (s_2-s_1)(W_{s_2}-W_{s_1})    \\
        &\hspace{2.0cm} - \tfrac{1}{2} (W_{s_2}-W_{s_1})^\top [\sigma^\top\!\Diff^{2}\!\nn\sigma] (x_1) (W_{s_2}-W_{s_1})     \\
        &\hspace{2.5cm} + \tfrac{1}{2} \trace[\sigma^\top\!\Diff^{2}\!\nn\sigma] (x_1) (s_2-s_1)
    \end{aligned}
\end{equation}
The penalization terms defined in \eqref{defLossSummand} satisfy
$ \lossSummand{t}{\unboundedTimeGrid} (\nn) = \mathcal{L}_{t,t+h} (\nn) \big(\Xc_{t}, \Xc_{t+h}\big) $ for $t\in\unboundedTimeGrid$.
We introduce two terms which allow us to distinguish whether $\XcExit$ is still due to occur. We set
´\begin{align}\label{defLossPrevPost}
	\lossSPre{t} (\nn) & \defined \mathcal{L}_{t\wedge\XcExit,\ (t+h)\wedge\XcExit} (\nn) \big(\Xc_{t\wedge\XcExit}, \Xc_{(t+h)\wedge\XcExit} \big) \notag \\
	\lossSPost{t} (\nn) & \defined \mathcal{L}_{t\vee\XcExit,\ (t+h)\vee\XcExit} (\nn) \big(\Xc_{t\vee\XcExit}, \Xc_{(t+h)\vee\XcExit} \big)
\end{align}
\begin{proposition}\label{prop:Identification_LossSummand_with_PrePost}
    Assume \hyperref[assumption:domain]{\emph{(Dom)}} and \hyperref[assumption:fbsde]{(FBSDE)} hold for the same $\gamma\in(0,1]$. 
    Let $\Psi \in L^2$. 
    Then there is a constant $C>0$ (not depending on $h$) such that for all $h\in (0,1)$ it holds that
    \begin{equation}
    	\E \big[ \Psi \loss^{\StoppD} (\nn) \big] 
    	\leq 4 \E \Big[ \Psi\!\! \sum_{t\in\unboundedTimeGrid\cap[0,\XcExit)}\! \big| \lossSPre{t} (\nn) \big|^2 \Big]
    	+ 4 \E \Big[ \Psi\!\! \sum_{t\in\unboundedTimeGrid\cap[0,\StoppD)}\! \1_{ \{ \XcExit < t+h \} }\ \big| \lossSPost{t} (\nn) \big|^2 \Big] + 2 C h^\frac{3}{2}. \closeEqn
    \end{equation}
\end{proposition}
\begin{proof}
    Observe by definition (see \eqref{defLossPrevPost}) on $ \{\XcExit\notin [t,t+h]\} $ only one of $\lossSPre{t}, \lossSPost{t}$ is non-zero.\footnote{
        Respectively $\{\XcExit\in [t,t+h]\}$ is the event (and thus the only (random) discrete-time step) where both loss types are non-zero; stretching along disjoint subsets of $[t,t+h]$. 
        For further details we refer to Lemma~\ref{lemma:generalLossSummandTriangle}.
    }
    Using Lemma~\ref{lemma:generalLossSummandTriangle} we obtain a random variable $R_t^{\XcExit}$ (see \eqref{def:R_LossTriangle}) that compensates the overlap; 
    \begin{equation}\label{eqn:lossSummand_split_with_overlap_compensation}
    	\lossSummand{t}{\unboundedTimeGrid} (\nn) = \1_{ \{ t < \XcExit \} }\ \lossSPre{t}(\nn) + \1_{ \{ \XcExit < t+h \} }\ \lossSPost{t}(\nn) + \1_{ \{ t < \XcExit < t+h \} }\ R_{t}^{\XcExit}
    \end{equation}
    and moreover there is a constant $C>0$ (not depending on $t\in\unboundedTimeGrid$ nor $h$) such that
    \begin{equation}\label{eqn:lossSummand_overlap_pathwise_bound}
    	|R_{t}^{\XcExit}|^2 \leq C \big( h^2 + \sup_{r,s\in[t,t+h]}\| W_s-W_r \|^{4} \big). 
    \end{equation}
    Next we use \cite[Corollary 1.3]{Seifried2024HoelderMoments}; a result on the moments of the local Hölder coefficient of a process.\footnote{
        The Brownian motion satisfies Kolmogorov's tightness criterion for all $p\geq 1$ with constant ratio $\tfrac{1}{2}$, i.e. 
        \begin{equation}
            \E \big[ \|W_t-W_s\|^p \big] \leq C_p |t-s|^\frac{p}{2} \quad\text{for all } t,s \geq 0
        \end{equation}
        where $C_p>0$ is a constant only depending on $p$.
        With this the prerequisites of \cite[Corollary 1.3]{Seifried2024HoelderMoments} are satisfied such that we can apply the result for the exponent $\vartheta = 8$ and tolerance $\varepsilon = \tfrac{1}{8}$ (using the notation of \cite{Seifried2024HoelderMoments}).
    } 
    Together with the Cauchy-Schwarz inequality this allows us to estimate
    %
    \begin{equation}
    	\begin{aligned}
    		\E \Big[ \Psi\!\! \sum_{t\in\unboundedTimeGrid\cap[0,\StoppD)} \1_{ \{\XcExit\in [t,t+h] \} } \sup_{r,s\in[t,t+h]}\| W_s-W_r \|^{4} \Big] 
    		& \leq \E \Big[ \Psi \sup_{\ell\in\nat_0} \sup_{ \substack{r,s\in[\ell h,(\ell+1)h] \\ r<s\leq\StoppD}} \| W_s-W_r \|^{4} \Big] \\
    		& \leq C \E \Big[ \sup_{\ell\in\nat_0} \sup_{ \substack{r,s\in[\ell h,(\ell+1)h]\\ r<s\leq\StoppD}} \| W_s-W_r \|^{8} \Big]^\frac{1}{2} \\
    		& \leq C h^\frac{3}{2} .
    	\end{aligned}
    \end{equation}
    Combine this with \eqref{eqn:lossSummand_split_with_overlap_compensation}, \eqref{eqn:lossSummand_overlap_pathwise_bound} and an elementary bound to conclude
    \begin{align}
    	\E \Big[ \Psi \sum_{t\in\unboundedTimeGrid\cap[0,\StoppD)} \big| \lossSummand{t}{\unboundedTimeGrid} (\nn) \big|^2 \Big]
        & \leq 4 \E \Big[ \Psi \sum_{t\in\unboundedTimeGrid\cap[0,\StoppD)} \1_{ \{ t < \XcExit \} }\ \big| \lossSPre{t} (\nn) \big|^2 \Big] \\ 
    	& \hspace*{1.0cm} + 4 \E \Big[ \Psi \sum_{t\in\unboundedTimeGrid\cap[0,\StoppD)} \1_{ \{ \XcExit < t+h \} }\  \big| \lossSPost{t} (\nn) \big|^2 \Big] + 2 C h^\frac{3}{2}. \qedhere
    \end{align}    
\end{proof}
We are in position to provide the proof of our main result. The argument is split into two parts. Firstly we derive a general bound in terms of the scaling variable $\Psi$ and the stepsize $h$; presented separately in \eqref{eqn:Proof_general_LossRate_Bound_1}, \eqref{eqn:Proof_general_LossRate_Bound_2} and \eqref{eqn:Proof_general_LossRate_Bound_3} below. Secondly we further estimate those bounds to extract the components which yield the rate of the loss functional. This is split into individual steps.
\begin{proof}[Proof of Theorem~\ref{thm:general_dynamical_LossRate}]
    From Proposition~\ref{prop:Identification_LossSummand_with_PrePost} we know that
	%
	\begin{equation}
		\E \big[ \Psi \loss^{\StoppD} (\nn) \big] 
		\leq 4 \E \Big[ \Psi\!\! \sum_{t\in\unboundedTimeGrid\cap[0,\XcExit)}\! \big| \lossSPre{t} (\nn) \big|^2 \Big]
		+ 4 \E \Big[ \Psi\!\! \sum_{t\in\unboundedTimeGrid\cap[0,\StoppD)}\! \1_{ \{ \XcExit < t+h \} }\ \big| \lossSPost{t} (\nn) \big|^2 \Big] + 2 C h^\frac{3}{2}
	\end{equation}
    Regarding the terms \textit{previous of first exit} $\XcExit$ we can exploit that the exact solution $u$ of the PDE \eqref{eqnPDE} is available and particularly satisfies the PDE. Thus we use an elementary bound to obtain
    \begin{equation}
    	\sum_{t\in\unboundedTimeGrid\cap[0,\XcExit)} \big| \lossSPre{t}(\nn) \big|^2 \leq 2 \sum_{t\in\unboundedTimeGrid\cap[0,\XcExit)} \big| \lossSPre{t}(\nn) - \lossSPre{t}(u) \big|^2 + 2 \sum_{t\in\unboundedTimeGrid\cap[0,\XcExit)} \big| \lossSPre{t}(u) \big|^2.
    \end{equation}
    From Lemma~\ref{lemma:proofTriangulusLossDistance} we obtain an estimate for the accumulated loss distance until the first exit $\XcExit$; 
    \begin{multline}\label{eqn:Proof_general_LossRate_Bound_1}
        \E \Big[ \Psi \sum_{ t\in\unboundedTimeGrid\cap[0,\XcExit) } \big| \lossSPre{t} (\nn)- \lossSPre{t} (u)\big|^2 \Big] 
        \leq C h^{2}\ \E \Big[ \Psi \Big(1+ \big\langle \E_{\circ} [\Psi] \big\rangle_{\XcExit^+}^3 \Big) \XcExit^+ \Big] \\
        + \|\nn-u\|_{\domain}^2\ C \E \bigg[ \Psi \Big( \big(1+ \big\langle \E_{\circ} [\Psi] \big\rangle_{\XcExit^+} \big)  
        + h \big(1+ \big\langle \E_{\circ} [\Psi] \big\rangle_{\XcExit^+}^2 \big)  
        + h \Big) \XcExit^+ \bigg]
    \end{multline}
    Using an integral representation of $\lossSPre{t} (u)$ (see \eqref{eqn:I_II_III_bound_splitt}) together with the bounds of Lemmas~\ref{lemma:proofLossRateExactMartingaleRate}, \ref{lemma:ProofPropLossRateExact} and \ref{lemma:ProofPropLossRateResidue} we obtain a random variable $\GirsanovK{}$ (depending on $\Psi$ and $\XcExitPlus$; see \eqref{Def:Girsanov_K});
	\begin{equation}\label{eqn:Proof_general_LossRate_Bound_2}
		\E \Big[ \Psi \sum_{t\in\unboundedTimeGrid\cap[0,\XcExit)} \big| \lossSPre{t} (u) \big|^2\Big] \leq C h^2 \bigg( h +  \E \big[ \GirsanovK{} \XcExitPlus \big] 
        + \E \Big[ \Psi \big(1 + \big\langle \E_{\circ}^{\prob} [ \Psi ] \big\rangle_{\XcExit^+} \big) \XcExitPlus\Big] 
        \bigg) 
	\end{equation}
    On the other hand regarding the \textit{posterior} loss sum we combine Lemma~\ref{lemma:proofTriangulusMartingaleBound} and Lemma~\ref{lemma:proofTriangulusMartingaleRate} to obtain a random variable $\GirsanovG{2}$ (depending on $\Psi$ and $\StoppD$; see \eqref{Def:Girsanov_G}) such that
    \begin{multline}\label{eqn:Proof_general_LossRate_Bound_3}
        \E \Big[ \Psi \sum_{t\in\unboundedTimeGrid\cap[0,\StoppD)} |\lossSPost{t} (\nn)|^2 \Big] \leq C \E [\Psi^2]^\frac{1}{2} \big( h^2 + h^3 \big) + C h\ \E \Big[ \GirsanovG{} (\StoppD - \XcExit^+) \Big] \\
        + C h \E \bigg[ \Psi \Big( h^2 + h + 1 + h^2\big\langle \E_{\circ} [\Psi] \big\rangle_{\StoppD} + h\big\langle \E_{\circ} [\Psi] \big\rangle_{\StoppD} + \big\langle \E_{\circ} [\Psi] \big\rangle_{\StoppD}^2 \Big) (\StoppD - \XcExit^+ ) \bigg]
    \end{multline}

    \step{1} We show 
    \begin{equation}
        \E \Big[ \Psi \sum_{ t\in\unboundedTimeGrid\cap[0,\XcExit) } \big| \lossSPre{t} (\nn)- \lossSPre{t} (u)\big|^2 \Big] \leq C h^2 + C \|\nn-u\|_{\domain}^2
    \end{equation}
    
    With the estimate \eqref{eqn:Proof_general_LossRate_Bound_1} we only have to verify that both expectations are finite. From Theorem~\ref{thm:PreliminaryResultsBGG} we know $\StoppD$ has a positive exponential moment and thus also $\XcExit$. Using Lemma~\ref{lemma:IntegrabilityQV_PsiGenericStopp}\footnote{
    	Applied to $\genericStopp = \XcExitPlus$, $V = \Psi$ and $p=3$.
    } 
    we obtain 
    $ \E [ \Psi (1+ \langle \E_{\circ} [\Psi] \rangle_{\XcExit^+}^3 ) \XcExit^+ ] < \infty $.
    Moreover with the Cauchy-Schwarz inequality and Lemma~\ref{lemma:IntegrabilityQV_PsiGenericStopp}\footnote{
        Applied (twice) for $p=1$ as well as $p=2$ with $\genericStopp = \XcExitPlus$ and $V = \Psi$.
    } 
    we compute
    \begin{align}
        & \E \bigg[ \Psi \Big( \big(1+ \big\langle \E_{\circ} [\Psi] \big\rangle_{\XcExit^+} \big)  
        + h \big(1+ \big\langle \E_{\circ} [\Psi] \big\rangle_{\XcExit^+}^2 \big)  
        + h \Big) \XcExit^+ \bigg] \\
        & \hspace*{1.0cm} = \E \Big[ \Psi \big(1+ \big\langle \E_{\circ} [\Psi] \big\rangle_{\XcExit^+} \big) \XcExitPlus \Big] + h \E \Big[ \Psi \big(1+ \big\langle \E_{\circ} [\Psi] \big\rangle_{\XcExit^+}^2 \big) \XcExitPlus \Big] + h \E \big[\Psi \XcExitPlus] \\
        & \hspace*{1.0cm} \leq C + C h + h \E \big[ \Psi^2 \big]^\frac{1}{2} \E \big[ {\XcExitPlus}^2 \big]^\frac{1}{2} < \infty \closeEqn
    \end{align}

    \step{2} We show
    \begin{equation}
        \E \Big[ \Psi \sum_{t\in\unboundedTimeGrid\cap[0,\XcExit)} \big| \lossSPre{t} (u) \big|^2\Big] \leq C h^2
    \end{equation}

    With the estimate \eqref{eqn:Proof_general_LossRate_Bound_2} we verify that the expectations in the bound are finite. As in the previous step with Theorem~\ref{thm:PreliminaryResultsBGG} we are allowed to invoke Lemma~\ref{lemma:IntegrabilityQV_PsiGenericStopp} and Lemma~\ref{lemma:IntegrabilityQV_GirsanovGG}. Firstly\footnote{
        Apply Lemma~\ref{lemma:IntegrabilityQV_PsiGenericStopp} to $\genericStopp = \XcExitPlus$, $V = \Psi$ and $p=1$.
    }
    we have
    $ \E [ (1 + \langle \E_{\circ}^{\prob} [ \Psi ] \rangle_{\XcExit^+} ) \XcExitPlus ] < \infty $.
    Secondly we recall the definition of $\GirsanovK{}$ (see \eqref{Def:Girsanov_K}) and apply Lemma~\ref{lemma:IntegrabilityQV_PsiGenericStopp} (as above) together with the second part of Lemma~\ref{lemma:IntegrabilityQV_GirsanovGG} to obtain\footnote{
        Apply Lemma~\ref{lemma:IntegrabilityQV_GirsanovGG} to $\genericStopp = \XcExitPlus$, $V = \Psi$ resp. $U = \Psi \big(1 + \big\langle \E_\circ^\prob [\Psi \big] \big\rangle_{\XcExitPlus} \big)$.
    }
    \begin{multline}
        \E \big[ \GirsanovK{} \XcExitPlus \big] = 
        \E \Big[ \Psi \big(1 + \big\langle \E_\circ^\prob [\Psi \big] \big\rangle_{\XcExitPlus} \big) \XcExitPlus \Big] \\
        + \E \bigg[ 
            \Psi \big(1 + \big\langle \E_\circ^\prob [\Psi \big] \big\rangle_{\XcExitPlus} \big) \Big\langle \E_{\circ} \big[ \Psi \big(1 + \big\langle \E_\circ^\prob [ \Psi \big] \big\rangle_{\XcExitPlus} \big) \big] \Big\rangle_{\XcExitPlus}^2 \XcExitPlus
        \bigg] < \infty \quad \closeEqn
    \end{multline}

    \step{3} We show
    \begin{equation}
        \E \Big[ \Psi \sum_{t\in\unboundedTimeGrid\cap[0,\StoppD)} |\lossSPost{t} (\nn)|^2 \Big] \leq C h^2 + C h \E \big[ (\StoppD - \XcExitPlus)^2 \big]^\frac{1}{2}
    \end{equation}
    Regarding \eqref{eqn:Proof_general_LossRate_Bound_3} we apply the Cauchy-Schwarz inequality to extract the $L^2$-norm of $\StoppD - \XcExit$ and then check that the resulting expectations are finite. For ease of presentation we omit the bound obtained from Cauchy-Schwarz and directly proceed with the individual components. 
    Recall the definition of $\GirsanovG{}$ (see \eqref{Def:Girsanov_G}) and use the first part of Lemma~\ref{lemma:IntegrabilityQV_GirsanovGG}\footnote{
        Applied to $\genericStopp = \StoppD$ and $V = \Psi$ resp. $U = \Psi \big(1 + \big\langle \E_\circ^\prob [\Psi \big] \big\rangle_{\StoppD} \big)$.
    } to obtain
    $ \E [ \GirsanovG{2} ] < \infty $. 
    Finally with Lemma~\ref{lemma:IntegrabilityQV_ClosedMartingale}\footnote{
        Applied (twice) for $p=8$ as wells as $p=16$ with $\genericStopp = \StoppD$ and $V = \Psi$.
    }
    and using elementary bounds we compute
   	\begin{align}
   		& \E \bigg[ \Psi^2 \Big( h^2 + h + 1 + h^2\big\langle \E_{\circ} [\Psi] \big\rangle_{\StoppD} + h\big\langle \E_{\circ} [\Psi] \big\rangle_{\StoppD} + \big\langle \E_{\circ} [\Psi] \big\rangle_{\StoppD}^2 \Big)^2 \bigg] \\
   		& \hspace*{1.0cm} \leq \tfrac{1}{2} \E [ \Psi^4 ] + \tfrac{1}{2} \E \bigg[ \Big( (h^2 + h + 1) + (h^2+h) \big\langle \E_{\circ} [\Psi] \big\rangle_{\StoppD} + \big\langle \E_{\circ} [\Psi] \big\rangle_{\StoppD}^2 \Big)^4 \bigg] \\
   		& \hspace*{1.0cm} \leq \tfrac{1}{2} \E [ \Psi^4 ] + 2^5 \bigg( 
   		(h^2+h+1)^4 + (h^2+h)^4 \E \Big[ \big\langle \E_{\circ} [\Psi] \big\rangle_{\StoppD}^4 \Big] + \E \Big[ \big\langle \E_{\circ} [\Psi] \big\rangle_{\StoppD}^{8} \Big] 
   		\bigg) < \infty.\ \ \qedhere
   	\end{align}
\end{proof}
We conclude with a consequence of our main result Theorem~\ref{thm:general_dynamical_LossRate} and the properties of the Euler-Maruyama scheme from \cite{BGG_forward_exit, Seifried2025bsdeDiscretization}; see Theorem~\ref{thm:PreliminaryResultsBGG} and Corollary~\ref{corollary:BSDE_Discretization_EM_Stopp}. We obtain bounds which only depend on $h$ and a universal approximation distance.\footnote{
    From Theorem~\ref{thm:general_LossRate} a bound of $\mathfrak{L}^\Psi$ (see \eqref{defLossFunctional_PSI}) is immediate. Our presentation is intended to offer additional insight.
}
\begin{theorem}[General Loss Rate]\label{thm:general_LossRate}
    Let $\Psi$ denote a random variable that is bounded away from zero and satisfies $\Psi \in L^{p}$ for some $p\geq 32$. 
    Assume \hyperref[assumption:domain]{\emph{(Dom)}} and \hyperref[assumption:fbsde]{(FBSDE)} hold for the same $\gamma \in(0,1]$ and assume \hyperref[assumption:exp_bsde]{(e-$1$)}, \hyperref[assumption:exp_EM]{(e-$2$)} and \hyperref[assumption:exp_EM_bound]{(e-$3$)} are satisfied.
    Then there are $C_1, C_2 >0$ (not depending on $h$); 
    \begin{equation}
        \E \big[ \Psi \lossBdry^{\StoppD} (\nn) \big] \leq C_1 h^\frac{1}{8} 
        \quad\text{and}\quad
        \E \big[\Psi \loss^{\bar\tau} (\nn) \big] \leq C_2 h^{1+\frac{1}{4}} + C_2 {\|U-u\|}_\domain^2 
        \quad\text{for all } h\leq h_0 \closeEqn
    \end{equation}
\end{theorem}
\begin{proof}
    Individually we derive the rate of the boundary resp. dynamical penalization term.
	
	\step{1} We show there is $C_1 >0$ such that 
    $ \E [ \Psi \lossBdry^{\StoppD} (\nn) ] \leq C_1 h^\frac{1}{8} $.
    
    The Lipschitz continuity\footnote{
        By assumption $\nn,g\in\HoelderSpace{3}{\gamma} (\overline{\domain})$, thus without restriction we may assume that both maps are globally Lipschitz. Otherwise one merely defines them on $\overline{\domain}$ and invokes the classical extension lemma, see e.g. \cite[Lemma 2.20]{LeeSmoothManifolds}.
    } of $\nn, g$ and the fact that $u$ solves \eqref{eqnPDE} (thus $u(X_\tau)=g(X_\tau)$) yield
    \begin{equation}
        \lossBdry^{\StoppD} (\nn) 
        \leq 2 \big| \nn (\X_{\StoppD}) - u(X_{\tau}) \big|^2 + 2 \big| g(X_{\tau}) - g(\X_{\StoppD}) \big|^2
        \leq 2 (\lipschitz{\nn}^2 + \lipschitz{g}^2) \|\Xc_{\StoppD}-X_{\tau}\|^2
    \end{equation}
    Using Cauchy-Schwarz, Corollary~\ref{corollary:BSDE_Discretization_EM_Stopp} and the subadditivity of the square-root there is $K_1 >0$;
    \begin{equation}
        E \big[ \Psi \lossBdry^{\StoppD} (\nn) \big] 
        \leq \E \big[ \Psi^2 \big]^{\frac{1}{2}}  \E \big[ \lossBdry^{\StoppD} (\nn)^2 \big]^\frac{1}{2}
        \leq K_1 h + K_1 \E \big[ |\tau - \StoppD|^6 \big]^\frac{1}{4} + K_1 \E \big[ |\tau - \StoppD|^2 \big]^\frac{1}{4}  
    \end{equation} 
    and with Theorem~\ref{thm:PreliminaryResultsBGG} we know there are $c_2, c_6 > 0$ such that
    \begin{equation}
        \E \big[ |\tau - \StoppD|^6 \big]^\frac{1}{4} + \E \big[ |\tau - \StoppD|^2 \big]^\frac{1}{4}  \leq c_6 h^\frac{1}{8} + c_2 h^\frac{1}{8} \closeEqn
    \end{equation}

    \step{2} We show there is $C_2 >0$ such that 
    $ \E [\Psi \loss^{\bar\tau} (\nn) ] \leq C_2 h h^\frac{1}{4} + C_2 {\|U-u\|}_\domain^2 $.

    From Theorem~\ref{thm:general_dynamical_LossRate} we obtain a constant $K_2>0$ such that
    \begin{equation}
        \E \big[\Psi \loss^{\bar\tau} (\nn) \big] \leq K_2 h^\frac{3}{2} + K_2 h \E \big[ (\StoppD - \XcExitPlus)^2 \big]^\frac{1}{2} + K_2 {\|U-u\|}_\domain^2
    \end{equation}
    Recall that by definition we have $\XcExit\leq\XcExitPlus\leq\StoppD$ and thus $0\leq\StoppD - \XcExitPlus \leq \StoppD - \XcExit$. Moreover $\X$ is also the Euler-Maruyama approximation of $\Xc$ and since \hyperref[assumption:exp_EM_bound]{(e-$3$)} is satisfied (for $\tau\vee\StoppD$) we particularly know that $\XcExit$ fulfills \hyperref[assumption:exp_EM]{(e-$2$)}. We apply Theorem~\ref{thm:PreliminaryResultsBGG} to obtain $\widetilde{c}_2 > 0 $ such that
    \[
        \E \big[ (\StoppD - \XcExitPlus)^2 \big]^\frac{1}{2} \leq \E \big[ (\XcExit - \StoppD)^2 \big]^\frac{1}{2} \leq \widetilde{c}_2 h^\frac{1}{4} . \qedhere
    \]
\end{proof}

\section{Auxiliary Results for Theorem~\ref{thm:general_dynamical_LossRate}}\label{sec:AuxResults}

Subsequently we provide the preparatory results required in Section~\ref{sec:GeneralLossRate}.\footnote{
    The results in this section are postulated under minimal assumptions on $\Psi$; particularly in general the bounds do not necessarily need to be finite. We refer to Section~\ref{sec:GeneralLossRate} for sufficient regularity conditions on $\Psi$.
}
Throughout this section we assume \hyperref[assumption:domain]{\emph{(Dom)}} and \hyperref[assumption:fbsde]{(FBSDE)} are satisfied with the same $\gamma\in(0,1]$ and assume there is $\rho>0$ such that \hyperref[assumption:exp_bsde]{(e-$1$)} holds. Particularly without further restriction we consider $h\in(0,1)$.
\begin{lemma}\label{lemma:proofTriangulusLossDistance}
    Let $\Psi \in L^1$ be bounded away from zero. There is $C>0$ (not depending on $h$);
    \begin{multline}
        \E \Big[ \Psi \sum_{ t\in\unboundedTimeGrid\cap[0,\XcExit) } \big| \lossSPre{t} (\nn)- \lossSPre{t} (u)\big|^2 \Big] 
        \leq C h^{2}\ \E \Big[ \Psi \big(1+ \big\langle \E_{\circ} [\Psi] \big\rangle_{\XcExitPlus}^3 \big) \XcExitPlus \Big] \\
        + C \|\nn-u\|_{\domain}^2\ \E \Big[ \Psi \Big( \big(1+ \big\langle \E_{\circ} [\Psi] \big\rangle_{\XcExitPlus} \big)  
        + h \big(1+ \big\langle \E_{\circ} [\Psi] \big\rangle_{\XcExitPlus}^2 \big)  
        + h \Big) \XcExitPlus \Big]
    \end{multline}
    where $\XcExitPlus$ denotes the first time on the grid $\unboundedTimeGrid$ after $\XcExit$, i.e.
    $ \XcExitPlus \defined \min \big\{ t\in\unboundedTimeGrid \,|\, \XcExit \leq t \big\} $. \close
\end{lemma}
\begin{proof}
	Fix $t\in\unboundedTimeGrid$. On the event\footnote{
		We could invoke a sufficiently smooth extension of $u$ and avoid to restrict on $\{t < \XcExit\}$.
		However in Lemmas~\ref{lemma:ProofPropLossRateExact} and \ref{lemma:proofLossRateExactMartingaleRate} below we use that $u$ solves \eqref{eqnPDE} and thus are forced to consider trajectories of $\Xc$ up to the first exit. 
	} 
	$\{t < \XcExit\}$ we successively introduce random variables $\mathcal{U}_t$, $\mathcal{V}_t$, $\mathcal{W}_t$ such that
	$ \lossSPre{t}(\nn) - \lossSPre{t}(u) = \mathcal{U}_t + \mathcal{V}_t + \mathcal{W}_t $.
	The first one;
	\begin{align*}
		\mathcal{U}_t \defined \nn (\Xc_{(t+h)\wedge\XcExit}) - \nn (\Xc_{t\wedge\XcExit}) - \big( u (\Xc_{(t+h)\wedge\XcExit}) - u (\Xc_{t\wedge\XcExit}) \big)
	\end{align*}
	can be expanded and estimated as follows
    \begin{equation}
        \begin{aligned}
            |\mathcal{U}_t| \leq & \; \big| \nn (\Xc_{(t+h)\wedge\XcExit}) - \nn (\Xc_{t\wedge\XcExit}) - \Diff\!\nn (\Xc_{t\wedge\XcExit})^\top (\Xc_{(t+h)\wedge\XcExit} - \Xc_{t\wedge\XcExit}) \\
    		&\hspace*{3.5cm} - \tfrac{1}{2} (\Xc_{(t+h)\wedge\XcExit} - \Xc_{t\wedge\XcExit})^\top \Diff^2\!\nn (\Xc_{t\wedge\XcExit}) (\Xc_{(t+h)\wedge\XcExit} - \Xc_{t\wedge\XcExit}) \big| \\
    		&\hspace*{0.5cm} + \big| u (\Xc_{(t+h)\wedge\XcExit}) - u(\Xc_{t\wedge\XcExit}) - \Diff\! u (\Xc_{t\wedge\XcExit})^\top (\Xc_{(t+h)\wedge\XcExit} - \Xc_{t\wedge\XcExit}) \\
    		&\hspace*{3.5cm} - \tfrac{1}{2} (\Xc_{(t+h)\wedge\XcExit} - \Xc_{t\wedge\XcExit})^\top \Diff^2\! u (\Xc_{t\wedge\XcExit}) (\Xc_{(t+h)\wedge\XcExit} - \Xc_{t\wedge\XcExit}) \big| \\
    		&\hspace*{1.0cm} + \Big|\big(\Diff\! u (\Xc_{t\wedge\XcExit}) - \Diff\! \nn (\Xc_{t\wedge\XcExit}) \big)^\top (\Xc_{(t+h)\wedge\XcExit} - \Xc_{t\wedge\XcExit}) \Big| \\
    		&\hspace*{1.5cm} + \tfrac{1}{2} \big| (\Xc_{(t+h)\wedge\XcExit} - \Xc_{t\wedge\XcExit})^\top \Diff^2\! u (\Xc_{t\wedge\XcExit}) (\Xc_{(t+h)\wedge\XcExit} - \Xc_{t\wedge\XcExit}) \big|   
        \end{aligned}
    \end{equation}
	Using the Taylor estimate \eqref{eqnTaylorSecondSpace} for the first two summands and a standard bound\footnote{ 
        For bounded linear $T$ and bounded bilinear $S$, we have $\|Tx\|\leq\|T\|\|x\|$ and $S(x,y) \leq \|S\|\|x\|\|y\|$.
    } 
    yields
	\begin{align}\label{eqnProofLossRateBoundU}
		|\mathcal{U}_t| & \leq C \| \Xc_{(t+h)\wedge\XcExit} - \Xc_{t\wedge\XcExit} \|^3 
		+ \sup_{y\in \bar\domain} \big\|\Diff\!\nn(y) - \Diff\! u(y)\big\|\ \|\Xc_{(t+h)\wedge\XcExit} - \Xc_{t\wedge\XcExit} \| \notag \\
		& \hspace*{1.0cm} + \sup_{y\in \bar\domain} \big\|\Diff^2\! \nn(y) - \Diff^2\! u(y)\big\|\ \| \Xc_{(t+h)\wedge\XcExit} - \Xc_{t\wedge\XcExit} \|^2  \\
		& \leq C \| \Xc_{(t+h)\wedge\XcExit} - \Xc_{t\wedge\XcExit} \|^3 + \|\nn-u\|_\domain \big( \| \Xc_{(t+h)\wedge\XcExit} - \Xc_{t\wedge\XcExit} \| + \| \Xc_{(t+h)\wedge\XcExit} - \Xc_{t\wedge\XcExit} \|^2 \big) \notag
	\end{align}
	Secondly,
	\begin{multline*}
		\mathcal{V}_t \defined f\big(\Xc_{t\wedge\XcExit}, \nn(\Xc_{t\wedge\XcExit}), \sigma(\Xc_{t\wedge\XcExit})^\top\Diff\!\nn(\Xc_{t\wedge\XcExit})\big) \big((t+h)\wedge\XcExit - t\wedge\XcExit \big)\\ 
		- f\big( \Xc_{t\wedge\XcExit}, u(\Xc_{t\wedge\XcExit}), \sigma( \Xc_{t\wedge\XcExit})^\top\Diff\! u( \Xc_{t\wedge\XcExit})\big) \big((t+h)\wedge\XcExit - t\wedge\XcExit \big) 
	\end{multline*}
	and since $f$ is Lipschitz and $\sigma$ is bounded, we have
	\begin{equation}\label{eqnProofLossRateBoundV}
		\begin{aligned}
			|\mathcal{V}_t| \leq & \; \lipschitz{f}\, h\, \Big(\big|\nn (\Xc_{t\wedge\XcExit})-u (\Xc_{t\wedge\XcExit})\big| + \big\|\sigma (\Xc_{t\wedge\XcExit})^\top\big(\Diff\! \nn(\Xc_{t\wedge\XcExit})-\Diff\! u(\Xc_{t\wedge\XcExit})\big)\big\| \Big) \\
			\leq & \; \lipschitz{f} \supremum{\sigma} \ h \ \|\nn -u\|_\domain .
		\end{aligned}
	\end{equation}
	Keeping the objective 
	$ \lossSPre{t}(\nn) - \lossSPre{t}(u) = \mathcal{U}_t + \mathcal{V}_t + \mathcal{W}_t $
	in mind, the third one has to be
    %
    \begin{equation}
        \begin{aligned}
            - \mathcal{W}_t & \defined \mathcal{U}_t + \mathcal{V}_t - \big(\lossSPre{t} (\nn) - \lossSPre{t} (u)\big)\\
			& = \big[ (\Diff\!\nn -\Diff\! u )^\top\sigma\big] (\Xc_{t\wedge\XcExit}) (W_{(t+h)\wedge\XcExit} - W_{t_{j}\wedge\XcExit}) \\
			& \hspace*{.5cm} + \big[\mu^\top (\Diff^2\!\nn -\Diff^2\! u)\sigma\big] (\Xc_{t\wedge\XcExit}) ((t+h)\wedge\XcExit - t\wedge\XcExit) (W_{(t+h)\wedge\XcExit} - W_{t\wedge\XcExit})\\
			& \hspace*{1.0cm} + \tfrac{1}{2} (W_{(t+h)\wedge\XcExit} - W_{t\wedge\XcExit})^\top \big[\sigma^\top (\Diff^2\!\nn - \Diff^2\! u )\sigma\big] (\Xc_{t\wedge\XcExit}) (W_{(t+h)\wedge\XcExit} - W_{t\wedge\XcExit})\\
			& \hspace*{1.5cm} - \tfrac{1}{2} \trace \big[ \sigma\sigma^\top (\Diff^2\!\nn -\Diff^2\! u) \big] (\Xc_{t\wedge\XcExit}) ((t+h)\wedge\XcExit - t\wedge\XcExit)
        \end{aligned}
    \end{equation}
	and with the Cauchy-Schwarz inequality for the Frobenius inner product
	we have
	\begin{align}\label{eqnProofLossRateBoundW}
		|\mathcal{W}_t| & \leq \supremum{\sigma}\, \big\| \Diff\!\nn (\Xc_{t\wedge\XcExit})-\Diff\! u (\Xc_{t\wedge\XcExit})\big\| \|W_{(t+h)\wedge\XcExit} - W_{t\wedge\XcExit}\|\notag\\
		& \hspace*{1.0cm} + \supremum{\mu}\supremum{\sigma}\, h\, \big\|\Diff^2\!\nn (\Xc_{t\wedge\XcExit})-\Diff^2\! u (\Xc_{t\wedge\XcExit})\big\| \| W_{(t+h)\wedge\XcExit} - W_{t\wedge\XcExit} \| \\
		& \hspace*{2.0cm} + \tfrac{1}{2}\supremum{\sigma}^2\, \big\|\Diff^2\!\nn (\Xc_{t\wedge\XcExit})-\Diff^2\! u (\Xc_{t\wedge\XcExit})\big\| \| W_{(t+h)\wedge\XcExit} - W_{t\wedge\XcExit} \|^2 \notag\\
		& \hspace*{3.0cm} + \tfrac{1}{2} \supremum{\sigma}^2\, h\, \big\|\Diff^2\!\nn (\Xc_{t\wedge\XcExit})-\Diff^2\! u (\Xc_{t\wedge\XcExit})\big\|\notag\\
		& \leq C \|\nn-u\|_{\domain} \Big( (1+h) \sup_{s\in[t,t+h]} \| W_{s} - W_{t} \| + \sup_{s \in [t,t+h]} \| W_{s} - W_{t} \|^2 + h \Big) \notag
	\end{align}
	We combine the bounds \eqref{eqnProofLossRateBoundU}, \eqref{eqnProofLossRateBoundV}, \eqref{eqnProofLossRateBoundW} and Proposition~\ref{prop:PropertiesOfStoppedEuler} to obtain
	\begin{multline}
		\big|\lossSPre{t}(\nn) - \lossSPre{t}(u)\big|^2 \leq C \sup_{s \in [t,t+h]} \| W_{s} - W_{t} \|^6 \\ 
		+ C \|U-u\|_{\domain}^2 \Big( \sup_{s \in [t,t+h]} \| W_{s} - W_{t}\|^2 + \sup_{s \in [t,t+h]} \| W_{s} - W_{t}\|^4 + h^2 \Big)
	\end{multline}
    Recall at the beginning $t\in\unboundedTimeGrid$ was chosen arbitrary and all calculations were carried out on the event $\{ t < \XcExit \}$. Whence we are in position to use \eqref{eqn:GirsanovBM} from Lemma~\ref{lemma:Girsanov}\footnote{
        Applied thrice (for $q=6,4,2$) to $\tau_1=0, \tau_2=\XcExitPlus$ and $\psi = \Psi $ explicitly using $\unboundedTimeGrid\cap[0,\XcExit) = \unboundedTimeGrid\cap[0,\XcExitPlus)$. 
    }
    and compute
	\begin{align}
		& \E \Big[ \Psi\!\! \sum_{ t\in\unboundedTimeGrid\cap[0,\XcExit) } \big| \lossSPre{t} (\nn)- \lossSPre{t} (u)\big|^2 \Big] \\
		& \hspace*{0.5cm} \leq C \E \Big[ \Psi\!\! \sum_{t\in\unboundedTimeGrid\cap[0,\XcExit)} \sup_{s\in[t,t+h]} \| W_{s}-W_{t} \|^6 \Big] 
		+ C \|\nn-u\|_{\domain}^2\ h^2 \E \Big[ \Psi\!\! \sum_{t\in\unboundedTimeGrid\cap[0,\XcExit)} 1 \Big] \\
		& \hspace*{1.0cm} + C \|\nn-u\|_{\domain}^2\  \E \Big[ \Psi\!\!  \sum_{t\in\unboundedTimeGrid\cap[0,\XcExit)} \sup_{s\in[t,t+h]} \| W_{s}-W_{t} \|^2 + \! \sup_{s\in[t,t+h]} \| W_{s}-W_{t} \|^4 \Big] \\
		& \hspace*{0.5cm} \leq C h^{\frac{6}{2}-1} \E \Big[ \Psi \big(1+ \big\langle \E_{\circ} [\Psi] \big\rangle_{\XcExit^+}^{3} \big) \XcExit^+ \Big] 
		+ C \|\nn-u\|_{\domain}^2\ h \E \big[ \Psi \XcExit^+ \big] \\
		& \hspace*{1.0cm} + C \|\nn-u\|_{\domain}^2 \bigg(\! \E \Big[ \Psi \big(1+ \big\langle \E_{\circ} [\Psi] \big\rangle_{\XcExit^+} \big) \XcExit^+ \Big] 
		+ h \E \Big[ \Psi \big(1+ \big\langle \E_{\circ} [\Psi] \big\rangle_{\XcExit^+}^{2} \big) \XcExit^+ \Big]\! \bigg). \qedhere
	\end{align}
\end{proof}
\begin{lemma}\label{lemma:proofTriangulusMartingaleBound}
	Let $\Psi \in L^1$ denote a random variable bounded away from zero.
    Consider 
    \begin{equation}\label{defProofLossRateMartingaleIntegrals}
		\postInt_t \defined \int_{t}^{t+h} \Diff\!\nn(\Xc_s)^\top\! \sigma(\Xc_{t}) \de W_s - [\Diff\!\nn^\top\! \sigma](\Xc_{t}) \big(W_{t+h} - W_{t} \big) \quad t\in\unboundedTimeGrid 
	\end{equation}
    then there is a constant $C>0$ (not depending on $h$) such that
	\begin{multline}
		\E \Big[ \Psi \sum_{t\in\unboundedTimeGrid\cap[0,\StoppD)} |\lossSPost{t} (\nn)|^2 \Big] \leq 
        C \E \Big[ \Psi \sum_{ t\in\unboundedTimeGrid\cap[\XcExit^+,\StoppD) } \big|\postInt_t \big|^2 \Big] + C \E [\Psi^2]^\frac{1}{2} \big( h^2 + h^3 \big) \\
        + C \E \bigg[ \Psi \Big( h^2+h+ h^2\big\langle \E_{\circ} [\Psi] \big\rangle_{\StoppD} + h\big\langle \E_{\circ} [\Psi] \big\rangle_{\StoppD}^2 \Big) (\StoppD - \XcExit^+ ) \bigg]. \closeEqn  
	\end{multline}
\end{lemma}
\begin{proof}
	Fix $t\in\unboundedTimeGrid$. With It\={o}'s formula we have
		\begin{multline}\label{eqnProofLossBoundItoRepresentationOfContinuation}
		\nn(\Xc_{(t+h)\vee\XcExit}) - \nn(\Xc_{t\vee\XcExit})  = \int_{t\vee\XcExit}^{(t+h)\vee\XcExit} \Diff\!\nn(\Xc_s)^\top \sigma (\Xc_{t}) \de W_s \\
		+ \int_{t\vee\XcExit}^{(t+h)\vee\XcExit} \Diff\!\nn(\Xc_s)^\top \mu (\Xc_{t}) + \tfrac{1}{2} \trace \big([\sigma\sigma^\top](\Xc_{t}) \Diff^2\! U(\Xc_s)\big) \de s
	\end{multline}
	and thus we rewrite the loss summand (cf. \eqref{defLossSummand_GeneralVersion} resp. \eqref{defLossSummand});
	\begin{equation}\label{eqn:AlternativeRepresentationLossPost}
		\begin{aligned}
			\lossSPost{t}(\nn) & = \int_{t\vee\XcExit}^{(t+h)\vee\XcExit} \Diff\!\nn(\Xc_s)^\top\! \mu (\Xc_{t}) + \tfrac{1}{2} \trace \big( [\sigma\sigma^\top](\Xc_{t}) \Diff^2\!\nn(\Xc_s) \big) \de s \\
			& \hspace*{0.5cm} + f\big(\Xc_{t\vee\XcExit}, \nn(\Xc_{t\vee\XcExit}), \sigma(\Xc_{t\vee\XcExit})^\top\! \Diff\!\nn(\Xc_{t\vee\XcExit})\big) ((t+h)\vee\XcExit-t\vee\XcExit) \\
			& \hspace*{1.0cm} - [\mu^\top\! \Diff^2\!\nn\sigma](\Xc_{t\vee\XcExit}) (W_{(t+h)\vee\XcExit} - W_{t\vee\XcExit}) ((t+h)\vee\XcExit-t\vee\XcExit) \\
			& \hspace*{1.5cm} -\tfrac{1}{2} (W_{(t+h)\vee\XcExit} - W_{t\vee\XcExit})^\top [\sigma^\top\! \Diff^2\!\nn\sigma](\Xc_{t\vee\XcExit}) (W_{(t+h)\vee\XcExit} - W_{t\vee\XcExit}) \\
			& \hspace*{2.0cm} + \tfrac{1}{2} \trace [\sigma^\top\! \Diff^2\!\nn\sigma] (\Xc_{t\vee\XcExit}) ((t+h)\vee\XcExit-t\vee\XcExit)  \\
			& \hspace*{2.5cm} + \postInt_{t}
		\end{aligned}
	\end{equation}
	We bound the first five lines on the event $\{ t<\StoppD \}$.\footnote{
        On said event we have $\{ \1_{ \{t<\StoppD\}} \Xc_{t} \}_{t\in\unboundedTimeGrid} \cup \{\Xc_{\XcExit} \} \subset \overline{\domain}\cup \{0\} $ and thus $\mu, \sigma, f, \Diff\nn$ and $\Diff^2\nn$ are bounded there.
    } 
    We compute
	\begin{equation}
		\Big|\int_{t\vee\XcExit}^{(t+h)\vee\XcExit} \Diff\!\nn(\Xc_s)^\top\! \mu (\Xc_{t}) + \tfrac{1}{2} \trace \big( [\sigma\sigma^\top](\Xc_{t}) \Diff^2\!\nn(\Xc_s) \big) \de s \Big| \leq C h
	\end{equation}
	where we additionally used that $\Diff\!\nn$ and $\Diff^2\!\nn$ are globally bounded. 
	Moreover, we have that
	\begin{equation}
		\Big( \big|f\big(\Xc_{t\vee\XcExit}, \nn(\Xc_{t\vee\XcExit}), [\sigma^\top\! \Diff\! \nn](\Xc_{t\vee\XcExit}) \big) \big| 
		+ \big| \trace [\sigma^\top\! \Diff^2\!\nn\sigma] (\Xc_{t\vee\XcExit}) \big| \Big)
		((t+h)\vee\XcExit-t\vee\XcExit) \leq C h
	\end{equation}
	and
	\begin{multline}
		\big|\tfrac{1}{2} (W_{(t+h)\vee\XcExit} - W_{t\vee\XcExit})^\top [\sigma^\top\! \Diff^2\!\nn\sigma](\Xc_{t\vee\XcExit}) (W_{(t+h)\vee\XcExit} - W_{t\vee\XcExit})\big| \\ 
		+ \big| [\mu^\top\! \Diff^2\!\nn\sigma](\Xc_{t\vee\XcExit}) (W_{(t+h)\vee\XcExit} - W_{t\vee\XcExit}) \big| ((t+h)\vee\XcExit-t\vee\XcExit) \\
		\leq C \big(\| W_{(t+h)\vee\XcExit} - W_{t\vee\XcExit} \|^2 + h\|W_{(t+h)\vee\XcExit} - W_{t\vee\XcExit}\|\big).
	\end{multline}
	Combine these observations with the fact that $\lossSPost{t}(\nn) = 0$ on $\{ \XcExit > t+h \}$ to obtain
	\begin{equation} 
		\big|\lossSPost{t} (\nn)\big| 
		\leq \1_{ \{ \XcExit \leq t+h \} } C \Big( h+ h \sup_{s\in[t,t+h]} \| W_{s} - W_{t} \|+ \sup_{s\in[t,t+h]}\| W_{s} - W_{t} \|^2 + |\postInt_{t}| \Big)
	\end{equation}
	\newpage
	With this we can bound the sum of \textit{post}-penalization terms
    %
    \begin{equation}
        \begin{aligned}
            \sum_{t\in\unboundedTimeGrid\cap[0,\StoppD)} \big|\lossSPost{t}(\nn)\big|^2 & \leq C \sum_{ t\in \unboundedTimeGrid\cap[\XcExit^+,\StoppD) } |\postInt_t|^2  
            + C \sum_{ t\in \unboundedTimeGrid\cap[\XcExit^+,\StoppD) } h^2 \\
    		& \hspace*{0.5cm} + C \sum_{ t\in \unboundedTimeGrid\cap[\XcExit^+,\StoppD) } h^2 \sup_{s\in[t, t+h]} \| W_{s} - W_{t} \|^2 + \sup_{s\in[t,t+h]} \| W_{s} - W_{t} \|^4 \\
    		& \hspace*{1.0cm} + C \big| \postInt_{\XcExit,\XcExitPlus}\big|^2 + C \big( h^2 + h^2\| W_{\XcExitPlus} - W_{\XcExit} \|^2 + \| W_{\XcExitPlus} - W_{\XcExit} \|^4 \big)
        \end{aligned}
    \end{equation}
	where $\postInt_{\XcExit,\XcExitPlus}$ is the quantity corresponding to \eqref{defProofLossRateMartingaleIntegrals} surrounding $\XcExit$, i.e.
	\begin{equation}
		\postInt_{\XcExit,\XcExitPlus} \defined \int_{\XcExit}^{\XcExitPlus}\!\! \Diff\!\nn(\Xc_s)^\top \Big(\sum_{\ell\in\unboundedTimeGrid} \1_{ \{ s\in(\ell,\ell+h] \} \cap \{ \XcExit\in [\ell,s)\} } \sigma (\Xc_{\ell}) \Big) \de W_s 
        - [\Diff\!\nn^\top\! \sigma] (\Xc_{\XcExit})(W_{\XcExitPlus}-W_{\XcExit}) 
	\end{equation}
	We conclude with bounds for the expectation of latter two lines of the preceding bound. We use the bound \eqref{eqn:GirsanovBM} from Lemma~\ref{lemma:Girsanov}\footnote{
        Applied to $\tau_1=\XcExitPlus, \tau_2=\StoppD$  and $\psi = \Psi $ where we are using that $\unboundedTimeGrid\cap[\XcExit,\StoppD) = \unboundedTimeGrid\cap[\XcExitPlus,\StoppD)$.
    }
    and obtain
	\begin{multline}\label{eqnProofIntrinsicRate_EQ1}
		\E \Big[ \Psi \sum_{t\in\unboundedTimeGrid\cap[\XcExitPlus, \StoppD)} h^2 \sup_{s\in[t, t+h]} \| W_{s} - W_{t} \|^2 + \sup_{s\in[t,t+h]} \| W_{s} - W_{t} \|^4 \Big] \\
		\leq C h^2 \E \Big[ \Psi \big(1+ \big\langle \E_{\circ} [\Psi] \big\rangle_{\StoppD} \big) (\StoppD - \XcExit^+ ) \Big] 
        + C h \E \Big[ \Psi \big(1+ \big\langle \E_{\circ} [\Psi] \big\rangle_{\StoppD}^2 \big) (\StoppD - \XcExit^+ ) \Big]
	\end{multline}
    Moreover we have that
	\begin{multline}\label{eqnProofTriangulus_EQ1}
		 \big|\postInt_{\XcExit,\XcExitPlus}\big|^4 \leq 2^3 \Big| \int_{\XcExit}^{\XcExitPlus} \Diff\!\nn(\Xc_s) \Big( \sum_{\ell\in\unboundedTimeGrid} \1_{ \{s\in(\ell,\ell+h]\} \cap \{ \XcExit\in [\ell,s) \} } \sigma(\Xc_{\ell}) - \sigma (\Xc_{\XcExit}) \Big) \de W_s \Big|^4 \\
		 + 2^3 \Big| \int_{\XcExit}^{\XcExitPlus} \big(\Diff\!\nn (\Xc_s) - \Diff\!\nn (\Xc_{\XcExit}) \big) \sigma(\Xc_{\XcExit}) \de W_s\Big|^4 
	\end{multline}
    We use Burkholder-Davis-Gundy (with $c_4>0$), the Lipschitz continuity of $\Diff\nn$ and $\XcExitPlus-\XcExit \leq h$;
	\begin{equation}
		\E \bigg[ \Big| \int_{\XcExit}^{\XcExitPlus} \big(\Diff\!\nn (\Xc_s) - \Diff\!\nn (\Xc_{\XcExit}) \big) \sigma(\Xc_{\XcExit}) \de W_s\Big|^4 \bigg] 
		\leq c_4 \lipschitz{\Diff\nn}^4 \supremum{\sigma}^4\ \E \Big[ h^2 \sup_{r\in[\XcExit, \XcExitPlus]} \| \Xc_r - \Xc_{\XcExit} \|^4 \Big]
	\end{equation}
	Similarly we estimate the first term of \eqref{eqnProofTriangulus_EQ1}, i.e. using the Burkholder-Davis-Gundy inequality and a standard bound of linear maps. We proceed with a bound of the corresponding remainder;
	\begin{equation}
		\E \bigg[ \Big( \int_{\XcExit}^{\XcExitPlus}\!\! \sum_{\ell\in\unboundedTimeGrid} \1_{ \{s\in (\ell,\ell+h]\}\cap\{\XcExit\in[\ell,s)\} } \|\sigma(\Xc_{\ell}) - \sigma (\Xc_{\XcExit}) \|^2 \de s \Big)^2 \bigg] \\
		\leq \lipschitz{\sigma}^4 h^2 \E \Big[ \sup_{r\in[\XcExit,\XcExitPlus]} \|\Xc_{r} - \Xc_{\XcExit} \|^4 \Big]
	\end{equation}
	Combine both and use the Cauchy-Schwarz inequality to conclude
	\begin{equation}
		\E \big[ \Psi |\postInt_{\XcExit,\XcExitPlus}|^2 \big]^2 \leq \E [\Psi^2] \E \big[ |\postInt_{\XcExit,\XcExitPlus}|^4 \big] \leq C \E [\Psi^2] h^2 \E \Big[ \sup_{r\in[\XcExit,\XcExitPlus]} \big\|\Xc_{r} - \Xc_{\XcExit} \big\|^4 \Big] \leq C \E [\Psi^2] h^4
	\end{equation}
    where the last step is due to Proposition \ref{prop:PropertiesOfStoppedEuler}, the strong Markov property of the Brownian motion and Doob's $L^4$-inequality; to be precise we have
	\begin{equation}
		\E \Big[ \sup_{r\in[\XcExit,\XcExitPlus]} \|\Xc_{r} - \Xc_{\XcExit} \|^4 \Big] 
		\leq C \E \Big[ h^4 +\!\! \sup_{r\in[\XcExit,\XcExitPlus]}\! \| W_{r} - W_{\XcExit} \|^4 \Big] 
		\leq C \E \Big[ h^4 + \sup_{r\in[0,h]} \| \widetilde{W}_{r} \|^4 \Big]
		\leq C h^2 
	\end{equation}
    Analogously using the Cauchy-Schwarz inequality and the strong Markov property we obtain
	\begin{equation}
		\E \Big[ \Psi \big( h^2 + h^2\| W_{\XcExitPlus} - W_{\XcExit} \|^2 + \| W_{\XcExitPlus} - W_{\XcExit} \|^4 \big) \Big] \leq C (h^2 + h^2 h + h^2) \E [\Psi^2]^\frac{1}{2} 
	\end{equation}
    Combine the latter two observations with \eqref{eqnProofIntrinsicRate_EQ1} to obtain the claim. 
\end{proof}

\begin{lemma}\label{lemma:proofTriangulusMartingaleRate}
    Let $\Psi \in L^4$ be bounded away from zero.
	There is $C>0$ (not depending on $h$) 
	and a random variable $\GirsanovG{}$ (see \eqref{Def:Girsanov_G} below)
	such that the sequence $\postInt$ (see \eqref{defProofLossRateMartingaleIntegrals}) satisfies
	\begin{equation}
		\E\Big[\Psi \sum_{ t\in \unboundedTimeGrid\cap[0,\StoppD) } \big|\postInt_t \big|^2\Big]
        \leq C h\ \E \Big[ h^2  \Psi \big(1 + \big\langle \E_\circ^\prob [ \Psi \big] \big\rangle_{\StoppD} \big) (\StoppD - \XcExit^+) + \GirsanovG{} (\StoppD - \XcExit^+) \Big]. \closeEqn
	\end{equation}	
\end{lemma}
\begin{proof}
   	The argument relies on Lemma~\ref{lemma:Girsanov}. In order to apply this result we rephrase $(\postInt_t)_{t\in\unboundedTimeGrid}$ as increments of a square integrable martingale. For this consider the predictable process
   	\begin{equation}
   		\begin{aligned}
   			H_s & \defined \sum_{\ell\in\unboundedTimeGrid} \1_{\{s\in (\ell ,\ell+h] \}} \big\{ \1_{\{ \XcExit < s \}} \big(\Diff\!\nn (\Xc_s) - \Diff\!\nn(\Xc_{\ell}) \big)^\top \sigma (\Xc_{\ell}) \\
   			& \hspace*{3.5cm} + \1_{\{ \ell < \XcExit \leq s \}} \big( [\Diff\!\nn^\top\sigma] (\Xc_{\ell}) -  [\Diff\!\nn^\top\sigma] (\Xc_{\XcExit}) \big) \big\} \quad s\geq 0.
   		\end{aligned} 
   	\end{equation}
   	Note that $H = 0$ on $\{ s < \XcExit \}$. Moreover since $\Diff\!\nn$ is globally bounded and $\sigma$ is continuous on $\overline{\domain}$ the process $H$ is bounded on $[0,\StoppD]$.
   	%
   	%
    Thus $ \M_t \defined \int_0^{t\wedge\StoppD} H_s \de W_s $, $t\geq 0$ 
   	defines a uniformly integrable martingale. By construction increments of $\M$ along $\unboundedTimeGrid$ coincide with $\postInt$ (see \eqref{defProofLossRateMartingaleIntegrals}), i.e.
   	$ \postInt_{t} = \M_{(t+h)\vee\XcExit^+} - \M_{t\vee\XcExit^+}$ for $t\in\unboundedTimeGrid$.
   	From the bound \eqref{eqn:GirsanovMart} of Lemma~\ref{lemma:Girsanov}\footnote{
        Applied to $\tau_1=\XcExit^+, \tau_2=\StoppD$ and $\psi = \Psi$.
    } 
    we obtain
   	\begin{equation}\label{lemma:proofTriangulusBoundSum1}
   		\E \Big[\Psi \sum_{ t\in\unboundedTimeGrid\cap[\XcExit^+,\StoppD) } \big|\M_{t+h} - \M_{t}\big|^2\Big] \leq C \E \Big[ \Psi \big(1 + \big\langle \E_\circ^\prob [ \Psi \big] \big\rangle_{\StoppD} \big) \int_{\XcExit^+}^{\StoppD} \|H_s\|^2\de s \Big]
   	\end{equation}
   	%
   	Since $\Diff\nn$ is Lipschitz, $\sigma$ is bounded on $\overline{\domain}$ we use Propostion~\ref{prop:PropertiesOfStoppedEuler} on $\{ t<\StoppD \}$ to compute
   	\begin{equation}\label{eqnProofLossRateExpectationOfIntegrandForMartingale}
		\int_{t}^{t+h} \|H_s\|^2\de s  
		\leq C\lipschitz{\Diff U}^2 \supremum{\sigma}^2  h\, \Big( h^2 + \sup_{s\in[t,t+h]} \|W_s - W_{t}\|^2 \Big).
   	\end{equation}
   	With \eqref{eqn:GirsanovBM} from Lemma~\ref{lemma:Girsanov} we bound the random sum of maximal Brownian increments\footnote{
        Applied to $\tau_1=\XcExit^+, \tau_2=\StoppD$ and $\psi = \Psi \big(1 + \big\langle \E_\circ^\prob [ \Psi \big] \big\rangle_{\StoppD} \big)$. Notice $\psi\in L^1$ by Cauchy-Schwarz and Lemma~\ref{lemma:IntegrabilityQV_ClosedMartingale}
    }
   	\begin{equation}
   		\E \Big[ \Psi \big(1 + \big\langle \E_\circ^\prob [ \Psi \big] \big\rangle_{\StoppD} \big) \sum_{t \in [\XcExit^+,\StoppD)} \sup_{s \in [t,t+h]}  \| W_s-W_{t} \|^2 \Big] \leq C \E \big[  \GirsanovG{} (\StoppD - \XcExit^+) \big] 
   	\end{equation}
   	where
   	\begin{equation}\label{Def:Girsanov_G}
   		\GirsanovG{} \defined \Psi \big(1 + \big\langle \E_\circ^\prob [\Psi \big] \big\rangle_{\StoppD} \big) \Big( 1 + \Big\langle \E_{\circ} \big[ \Psi \big(1 + \big\langle \E_\circ^\prob [ \Psi \big] \big\rangle_{\StoppD} \big) \big] \Big\rangle_{\StoppD} \Big)
   	\end{equation}
   	\newpage	
   	Combine this with \eqref{eqnProofLossRateExpectationOfIntegrandForMartingale} to conclude
    \[
        \E \Big[ \Psi\!\!\sum_{ t\in\unboundedTimeGrid\cap[\XcExit^+,\StoppD) }\!\! |\M_{t+h} - \M_{t}|^2\Big] \leq 
        C h \E \Big[ h^2  \Psi \big(1 + \big\langle \E_\circ^\prob [ \Psi \big] \big\rangle_{\StoppD} \big) (\StoppD - \XcExit^+) + \GirsanovG{} (\StoppD - \XcExit^+) \Big] . \qedhere
    \]
\end{proof}
%
%
%
We derive an alternative representation of $\lossSPre{}(u)$ (cf. \eqref{eqn:AlternativeRepresentationLossPost} for $\lossSPost{}(\nn)$). Utilizing the exact solution $u$ and that $\Xc$ only appears up to the first exit, we achieve a better rate.
Towards this we use exactly the same diagonalization argument as in the proof of \cite[Lemma A.1]{Knochenhauer2021ConvergenceRates}. Let $t\in\unboundedTimeGrid$;
\begin{equation}
	\begin{aligned}
		\Delta M_t^{\XcExit} & \defined \int_{t\wedge\XcExit}^{(t+h)\wedge\XcExit} \big( [\Diff^2\! u\sigma](\Xc_{t\wedge\XcExit})(W_s-W_{t\wedge\XcExit}) \big)^\top \sigma(\Xc_{t\wedge\XcExit}) \de W_s \\
		& = \tfrac{1}{2} (W_{(t+h)\wedge\XcExit}-W_{t\wedge\XcExit})^\top [\sigma^\top\! \Diff^2\! u\sigma](\Xc_{t\wedge\XcExit})(W_{(t+h)\wedge\XcExit}-W_{t\wedge\XcExit}) \\
		& \hspace*{0.5cm} - \tfrac{1}{2} \trace [ \sigma^\top\! \Diff^2\! u\sigma](\Xc_{t\wedge\XcExit}) ( (t+h)\wedge\XcExit - t\wedge\XcExit)
	\end{aligned}
\end{equation}
and thus we can express
\begin{equation}
	\begin{aligned}
		\lossSPre{t}(u) & = u(\Xc_{(t+h)\wedge\XcExit}) - u(\Xc_{t\wedge\XcExit}) 
		- \Delta M_t^{\XcExit}\\
		& \hspace*{0.5cm} + f \big(\Xc_{t\wedge\XcExit}, u(\Xc_{t\wedge\XcExit}), [\sigma^\top\! \Diff\! u] (\Xc_{t\wedge\XcExit}) \big) ( (t+h)\wedge\XcExit - t\wedge\XcExit ) \\
		& \hspace*{1.0cm} - \Big[ [\Diff\! u^\top\! \sigma] + ((t+h)\wedge\XcExit - t\wedge\XcExit) [\mu^\top\! \Diff^2\! u\sigma] \Big] (\Xc_{t\wedge\XcExit})\ (W_{(t+h)\wedge\XcExit}-W_{t\wedge\XcExit}) \\
	\end{aligned}
\end{equation}
Analogously to \eqref{eqnProofLossBoundItoRepresentationOfContinuation} with It\=o's formula we have 
\begin{equation}
	\begin{aligned}
		u(\Xc_{(t+h)\wedge\XcExit}) - u(\Xc_{t\wedge\XcExit}) & = \int_{t\wedge\XcExit}^{(t+h)\wedge\XcExit} \Diff\! u(\Xc_s)^\top\! \sigma (\Xc_{t}) \de W_s \\
		& \hspace*{0.5cm} + \int_{t\wedge\XcExit}^{(t+h)\wedge\XcExit} \Diff\! u(\Xc_s)^\top\! \mu (\Xc_{t}) + \tfrac{1}{2} \trace \big( [\sigma\sigma^\top](\Xc_{t}) \Diff^2\! u(\Xc_s)\big) \de s
	\end{aligned}
\end{equation}
and especially since $u$ solves the PDE \eqref{eqnPDE} it holds that
\begin{equation}
	f \big(\Xc_{t\wedge\XcExit}, u(\Xc_{t\wedge\XcExit}), [\sigma^\top\! \Diff\! u] (\Xc_{t\wedge\XcExit}) \big) ( (t+h)\wedge\XcExit - t\wedge\XcExit ) = \int_{t\wedge\XcExit}^{(t+h)\wedge\XcExit} -\InfGen [u](\Xc_{t\wedge\XcExit}) \de s.
\end{equation}
Combine the above to establish $\lossSPre{t} (u) = I_t+II_t+III_t$ in terms of an integral representation;
\begin{equation}\label{eqn:I_II_III_bound_splitt}
	\begin{aligned}
		\hspace*{-0.2cm}
		I_t + II_t + III_t 
		& \defined \int_{t\wedge\XcExit}^{(t+h)\wedge\XcExit} \Diff\! u(\Xc_s)^\top \sigma (\Xc_{t}) - \Diff\! u(\Xc_{t\wedge\XcExit})^\top\sigma(\Xc_{t\wedge\XcExit}) \\
		& \hspace*{0.5cm} -\Diff^2\! u(\Xc_{t\wedge\XcExit})\big(\mu(\Xc_{t\wedge\XcExit})h+\sigma(\Xc_{t\wedge\XcExit}) (W_s-W_{t\wedge\XcExit}) \big)^\top \sigma(\Xc_{t\wedge\XcExit}) \de W_s \\
		& \hspace*{0.25cm} + \int_{t\wedge\XcExit}^{(t+h)\wedge\XcExit} \Diff\! u(\Xc_{s})^\top\mu(\Xc_{t}) - [\Diff\! u^\top\mu] (\Xc_{t\wedge\XcExit}) \\
		& \hspace*{0.5cm} + \tfrac{1}{2} \Big(\trace \big( [\sigma\sigma^\top](\Xc_{t})\Diff^2\! u(\Xc_{s}) \big) - \trace [\sigma\sigma^\top\Diff^2\! u] (\Xc_{t\wedge\XcExit}) \Big) \de s  \\
		& \hspace*{0.25cm} + \big(h - ((t+h)\wedge\XcExit-t\wedge\XcExit) \big) \int_{t\wedge\XcExit}^{(t+h)\wedge\XcExit} [\mu^\top\Diff^2\! u\sigma] (\Xc_{t\wedge\XcExit}) \de W_s
	\end{aligned}
\end{equation}
where the identification is to be understood per integral. With an elementary bound we have
$ |\lossSPre{t} (u)|^2 \leq 4 ( |I_t|^2 + |II_t|^2 + |III_t|^2 ) $ 
and within the following results we further bound this. 
%
\begin{lemma}\label{lemma:proofLossRateExactMartingaleRate}
    Let $\Psi \in L^4$ denote a random variable bounded away from zero.
	There is a constant $C>0$ (not depending on $h$) and a random variable $\GirsanovK{}$ (see \eqref{Def:Girsanov_K}) such that
	\begin{equation}
		\E \Big[ \Psi \sum_{t\in\unboundedTimeGrid\cap[0,\XcExit)} \big|I_t\big|^2\Big] \leq C h^2 \Big( h +  \E \big[ \GirsanovK{} \XcExitPlus \big] \Big). \closeEqn
	\end{equation}
\end{lemma}
\begin{proof}
    Let $t\in\unboundedTimeGrid$. By definition (see \eqref{eqn:I_II_III_bound_splitt}) on ${\{t<\XcExit\}}$ it holds that
	\begin{equation}
		\begin{aligned}
			I_t & = \int_{t\wedge\XcExit}^{(t+h)\wedge\XcExit} \big\{ \Diff\! u(\Xc_{s}) - \Diff\! u (\Xc_{t\wedge\XcExit}) \\
			& \hspace*{2.5cm} - \Diff^2\! u(\Xc_{t\wedge\XcExit}) \big(\mu(\Xc_{t\wedge\XcExit})h + \sigma(\Xc_{t\wedge\XcExit}) (W_s-W_{t\wedge\XcExit}) \big) \big\}^\top \sigma (\Xc_{t\wedge\XcExit}) \de W_s
		\end{aligned}
	\end{equation}
	Consider the previsible process
	\begin{equation}
		\begin{aligned}
			H_s & \defined \sum_{\ell\in\unboundedTimeGrid} \1_{\{s\in(\ell, \ell+h]\} \cap \{s<\XcExit\}} \big\{ \Diff\! u(\Xc_s)-\Diff\! u(\Xc_{{\ell}}) \\
			& \hspace*{4.0cm} - \Diff^2\! u(\Xc_{{\ell}})\big(\mu(\Xc_{{\ell}})h+\sigma(\Xc_{{\ell}}) (W_s-W_{\ell}) \big) \big\}^\top \sigma(\Xc_{{\ell}}) \quad s\geq 0
		\end{aligned}
	\end{equation}
	Observe $H$ is bounded, thus $ \M \defined \int_0^{\cdot} H_s \de W_s $ is a uniformly integrable martingale and  
	%
	by choice $I_t = \M_{t+h} - \M_{t}$. With the bound \eqref{eqn:GirsanovMart} from Lemma~\ref{lemma:Girsanov};\footnote{
        Applied to $\tau_1=0, \tau_2=\XcExit$ and $\psi = \Psi $.
    }
	\begin{equation}
		\E \Big[ \Psi\!\! \sum_{t\in\unboundedTimeGrid\cap[0,\XcExit)}\!\! | I_t |^2 \Big] = \E \Big[ \Psi\!\! \sum_{t\in\unboundedTimeGrid\cap[0,\XcExit^+)}\!\! | \M_{t+h}-\M_{t} |^2 \Big] 
        \leq C \Big[ \Psi \big(1 + \big\langle \E_{\circ}^{\prob} [ \Psi ] \big\rangle_{\XcExit^+} \big) \int_{0}^{\XcExitPlus} \| H_s \|^2 \de s \Big]
	\end{equation}
	We use the Taylor bound \eqref{eqnTaylorSpaceDeriv}, Proposition~\ref{prop:PropertiesOfStoppedEuler} and the fact that $\Diff^2\!u, \sigma, \mu$ are bounded on $\overline{\domain}$ to compute on the event $\{t\in\unboundedTimeGrid\cap[0,\XcExitPlus)\}$;
	\begin{equation}
		\begin{aligned}
			\int_{t}^{t+h} \|H_s\|^2 \de s & \leq 2\supremum{\sigma}^2 \int_{t}^{(t+h)\wedge\XcExit} \big\|\Diff\! u(\Xc_s)-\Diff\! u(\Xc_{t})-\Diff^2\! u(\Xc_{t})(\Xc_s-\Xc_{t})\big\|^2 \\
			&  \hspace*{4.0cm} +\big\|\Diff^2\! u(\Xc_{t})\mu(\Xc_{t})\big((t+h)\wedge\XcExit-s\big)\big\|^2\de s \\
			& \leq C \int_{t}^{(t+h)\wedge\XcExit} h^2 + \|\Xc_s-\Xc_{t}\|^4 \de s\\
			& \leq Ch \big(h^2 + \sup_{s\in[t,t+h]}\| W_s- W_{t}\|^4\big)
		\end{aligned}
	\end{equation}
	We apply the bound \eqref{eqn:GirsanovBM} from Lemma~\ref{lemma:Girsanov}\footnote{
        Applied to $\tau_1=0, \tau_2 = \XcExit^+$ and $\psi = \Psi \big\langle \E_{\circ}^{\prob} [ \Psi ] \big\rangle_{\XcExit^+}$. Where $\psi\in L^1$ due to Cauchy-Schwarz and Lemma~\ref{lemma:IntegrabilityQV_ClosedMartingale}.
    }
    and obtain
	\begin{equation*}
		\E \Big[ \Psi \big(1 + \big\langle \E_{\circ}^{\prob} [ \Psi ] \big\rangle_{\XcExit^+} \big) \sum_{t\in\unboundedTimeGrid\cap [0,\XcExitPlus)} \sup_{s \in [t,t+h]}  \| W_s-W_{t} \|^4 \Big] \leq C h \E \big[ \GirsanovK{} \XcExitPlus \big]
	\end{equation*}
    where (cf. $\GirsanovG{}$ see \eqref{Def:Girsanov_G})
    \begin{align}\label{Def:Girsanov_K}
        \GirsanovK{} \defined \Psi \big(1 + \big\langle \E_\circ^\prob [\Psi \big] \big\rangle_{\XcExitPlus} \big) \Big( 1 + \Big\langle \E_{\circ} \big[ \Psi \big(1 + \big\langle \E_\circ^\prob [ \Psi \big] \big\rangle_{\XcExitPlus} \big) \big] \Big\rangle_{\XcExitPlus}^2 \Big) 
    \end{align}
    Combine everything to conclude the claim. 
\end{proof}
\begin{lemma}\label{lemma:ProofPropLossRateExact}
    Let $\Psi \in L^1$ be bounded away from zero.
	There is $C>0$ (not depending on $h$);
	\begin{equation*}
		\E \Big[ \Psi \sum_{t\in\unboundedTimeGrid\cap[0,\XcExit)}\big|II_t\big|^2\Big] \leq Ch^2 \E \Big[ \Psi \big(1 + \big\langle \E_{\circ}^{\prob} [ \Psi ] \big\rangle_{\XcExit^+} \big) \XcExitPlus\Big]. \closeEqn
	\end{equation*}    
\end{lemma}
\begin{proof}
	Let $t\in\unboundedTimeGrid$ and $s\in[t,t+h]$. On $\{s\leq\XcExit\}$ with the Lipschitz continuity of $\Diff\! u$ we have
	\begin{equation}
		\big| \Diff\! u(\Xc_s) -\Diff\! u(\Xc_{t} )\big)^\top\! \mu(\Xc_{t}) \big| \leq \supremum{\mu}\lipschitz{\Diff\! u} \| \Xc_{s}-\Xc_{t} \|
	\end{equation}
	and using the trace as inner product on quadratic matrices we estimate
	\begin{align}
		\big| \trace \big( [\sigma\sigma^\top] (\Xc_{t}) ( \Diff^2\! u(\Xc_s)-\Diff^2\! u(\Xc_{t}) ) \big) \big| \leq \supremum{\sigma}^2 \lipschitz{\Diff^2\! u} \| \Xc_{s}-\Xc_{t} \|.
	\end{align}
	Using Cauchy-Schwarz in combination with the previous estimates yields 
	\begin{equation}
			|II_t|^2 
			\leq C ((t+h)\wedge\XcExit-t\wedge\XcExit) \int_{t\wedge\XcExit}^{(t+h)\wedge\XcExit}\!\! \| \Xc_{s}-\Xc_{t} \|^2 \de s 
			\leq C h^2 \big( h^2 + \sup_{r\in [t,t+h]} \|W_r - W_{t}\|^2 \big)
	\end{equation}
	where the last step is due to Proposition~\ref{prop:PropertiesOfStoppedEuler}. With the bound \eqref{eqn:GirsanovBM} of Lemma~\ref{lemma:Girsanov}\footnote{
        Applied to $\tau_1=0, \tau_2=\XcExitPlus$ and $\psi = \Psi$.
    }
    we conclude 
	\begin{align}
		\E \Big[ \Psi \sum_{t\in\unboundedTimeGrid\cap[0,\XcExitPlus)} \sup_{s \in [t,t+h]}  \| W_{s}-W_{t} \|^2 \Big] & \leq C \E \Big[ \Psi \big(1 + \big\langle \E_{\circ}^{\prob} [ \Psi ] \big\rangle_{\XcExit^+} \big) \XcExitPlus\Big]. \qedhere
	\end{align}
\end{proof}
\begin{lemma}\label{lemma:ProofPropLossRateResidue}
    Let $\Psi \in L^1$ be bounded away from zero.
	There is $C>0$ (not depending on $h$);
	\begin{equation}
		\E \Big[ \exp(C\bar\tau) \sum_{t\in\unboundedTimeGrid\cap[0,\XcExit)}\big|III_t\big|^2\Big] \leq Ch^2 \E \Big[ \Psi \big(1 + \big\langle \E_{\circ}^{\prob} [ \Psi ] \big\rangle_{\XcExit^+} \big) \XcExitPlus\Big]. \closeEqn
	\end{equation}    
\end{lemma}
\begin{proof}
	Let $t\in\unboundedTimeGrid$. Observe that we have
	$| III_t|^2 \leq h^2 \supremum{\mu}^2\supremum{\Diff^2\! u}^2\supremum{\sigma}^2 \| W_{(t+h)\wedge\XcExit} - W_{t\wedge\XcExit}\|^2 $
	and thus we obtain the claim using the bound \eqref{eqn:GirsanovBM} from Lemma\ref{lemma:Girsanov} as in the previous proof.
\end{proof}
\bibliographystyle{plain}

\appendix
\section{Supplements}\label{sec:CombinedSupplements}
%
\begin{definition}[Hölder Space]\label{Definition Hoelder Space}
	Let $k \in \nat_0, \gamma \in (0,1]$. The Hölder space $\HoelderSpace{k}{\gamma} (\domain)$ is given by
	\begin{equation}
		\HoelderSpace{k}{\gamma} (\domain) \defined \big\{f\in\C^k(\domain) \,\big|\, \|f\|_{ \HoelderSpace{k}{\gamma} (\domain)}<\infty\big\}
	\end{equation}
	where
	\begin{equation*}
		\|f\|_{ \HoelderSpace{k}{\gamma} (\domain) } \defined \sum_{|\beta|\leq k} {\|D^\beta f\|}_\infty + \sum_{|\beta|=k} \langle D^\beta f\rangle_\gamma\quad\text{and}\quad \langle h \rangle_\gamma \defined \sup_{x,y\in\domain} \frac{|h(x)-h(y)|}{\|x-y\|^\gamma}. \closeEqn
	\end{equation*}
\end{definition}
\begin{lemma}
	For any $U\colon\overline{\domain}\rightarrow\R$ of class $\C^3$ there is a constant $C>0$ such that
	\begin{align}
		\bigl|U(x_1)-U(x_2)-\Diff\! U(x_2)^\top(x_1-x_2) \hspace*{4.0cm} & \\
		- \frac{1}{2}(x_1-x_2)^\top\Diff^2\! U(x_2)(x_1-x_2)\bigr| & \leq C\|x_1-x_2\|^3 \label{eqnTaylorSecondSpace} \\
		\bigl\|\Diff\! U(x_1)-\Diff\! U(x_2) - \Diff^2\! U(x_2)(x_1-x_2)\bigr\| &\leq C\|x_1-x_2\bigr\|^2 \label{eqnTaylorSpaceDeriv} 
	\end{align}
	for all $x_1,x_2\in\overline{\domain}\subset\R^d$.\close
\end{lemma}
\begin{proof}
	For a proof we refer to e.g.\ \cite[Lemma~A.4]{Knochenhauer2021ConvergenceRates}.
\end{proof}
\begin{lemma}[Dynamical loss triangle]\label{lemma:generalLossSummandTriangle}
	Let $\domain$ be a bounded domain, $\gamma\in(0,1]$ and $U\in\HoelderSpace{3}{\gamma} (\overline{\domain})$. Assume the maps $\mu, \sigma$ are bounded, $f$ is continuous and $\sigma$ is Lipschitz on $\overline{\domain}$. Let $x_1,x_2,x_3\in\overline{\domain}$ and $s_1<s_2$. Then for $s_3\in[s_1,s_2]$ there is a random variable $R = R (s_1,s_2,s_3;x_1,x_3)$ such that 
	\begin{equation}
		\mathcal{L}_{s_1,s_2} (\nn) (x_1,x_2) = \mathcal{L}_{s_1,s_3} (\nn) (x_1,x_3) + \mathcal{L}_{s_3,s_2} (\nn) (x_3,x_2) + R.
	\end{equation}
	Particularly for each $h\in(0,1)$, along paths of the interpolation $\Xc$ of the Euler-Maruyama scheme (see \eqref{def:EulerMaruyamaContinuousInterpolation}) for $t\in\unboundedTimeGrid$ on $\{\XcExit\in[t,t+h]\}$ the random variable $R \defined R(t,\XcExit, t+h; \Xc_{t}, \Xc_{\XcExit} )$ satisfies
	\begin{equation}
		\big| R(t,\XcExit, t+h; \Xc_{t}, \Xc_{\XcExit} ) \big|^2 \leq C \big( h^2 + \sup_{r,s\in[t,t+h]} \|W_s-W_r\|^4). \closeEqn
	\end{equation} 
\end{lemma}
\begin{proof}
	For ease of presentation we introduce the following notation $ f_2\defined f (\cdot, \nn, \sigma^\top\! \Diff\! \nn) $ and
	\begin{equation}
		f_3 \defined \Diff\! \nn^\top\sigma, \quad 
		f_4\defined\mu^\top\! \Diff^2\!\nn\sigma, \quad 
		f_5\defined \tfrac{1}{2}\sigma^\top\! \Diff^2\!\nn\sigma, \quad 
		f_6=\tfrac{1}{2} \trace[\sigma^\top\! \Diff^2\!\nn\sigma]
	\end{equation}
	where the index of each map corresponds to a line of $\mathcal{L}_{s_1,s_2} (\nn)$ (see \eqref{defLossSummand_GeneralVersion}). Set
	\begin{align}\label{def:R_LossTriangle}
		& R \defined \big(f_2(x_1)-f_2(x_3)\big) (s_2-s_3) - \big( f_3(x_1)-f_3(x_3) \big)(W_{s_2}-W_{s_3}) \notag \\
		& \hspace*{1.0cm} - \Big( \big(f_4(x_1)-f_4(x_3)\big)(s_2-s_3) + f_4(x_1)(s_3-s_1) \Big)(W_{s_2}-W_{s_3}) \notag \\
		& \hspace*{8.0cm}  - f_4(x_1)(s_2-s_3)(W_{s_3}-W_{s_1}) \notag \\
		& \hspace*{1.5cm} -  (W_{s_2}-W_{s_3})^\top \Big(\big(f_5(x_1)-f_5(x_3)\big)(W_{s_2}-W_{s_3}) + f_5(x_1)(W_{s_3}-W_{s_1})\Big) \notag \\
		& \hspace*{8.0cm} - (W_{s_3}-W_{s_1})^\top f_5(x_1)(W_{s_2}-W_{s_3}) \notag \\
		& \hspace*{2.5cm} + \big(f_6(x_1)-f_6(x_3)\big)(s_2-s_3).
	\end{align}
	With a straightforward computation we observe that
	\begin{equation}
		\mathcal{L}_{s_1,s_2} (\nn) (x_1,x_2) = \mathcal{L}_{s_1,s_3} (\nn) (x_1,x_3) + \mathcal{L}_{s_3,s_2} (\nn) (x_3,x_2) + R
	\end{equation}
	and the first part of the claim is established. 
	The second part is a consequence of Proposition~\ref{prop:PropertiesOfStoppedEuler}. We use that $f_3$ is Lipschitz and that $f_2,f_4,f_5,f_6$ are bounded on $\overline{\domain}$ to obtain
	\[
		\big|R\big|^2 \leq C \Big( h^2 + \|\Xc_{t}-\Xc_{\XcExit} \|^2 \|W_{t+h}-W_{\XcExit}\|^2 + \sup_{r,s\in[t,t+h]} \|W_s-W_r\|^4 \Big). \qedhere
	\]
\end{proof}
\subsection*{Bounds based on Girsanov's Theorem}

\begin{lemma}\label{lemma:Girsanov}   
	Let $(H_s)_{s\geq0}$ be predictable such that $ \E [\int_0^\infty H_s^2 \de s ] < \infty $ and set 
	$ \M \defined \int_0^\cdot H_s \de W_s $.
	Let $\tau_1\leq\tau_2$ denote $\unboundedTimeGrid$-valued stopping times with $\tau_2 \in L^1$. Let $\psi$ be bounded away from zero and satisfy $\psi\in L^p$ for some $p\geq 1$.
	Then there is $\ell > 0$ (only depending on the lower bound of $\psi$) such that
	\begin{equation}\label{eqn:GirsanovMart}
		\E \Big[ \psi\!\! \sum_{ t\in\unboundedTimeGrid\cap[\tau_1,\tau_2) }\!\! |\M_{t+h} - \M_{t}|^2 \Big] \leq 2 \E \bigg[  \psi \big( 1 +\ell^{-2}  \big\langle\E_{\circ} [ \psi ] \big\rangle_{\tau_2} \big) \int_{\tau_1}^{\tau_2} \|H_r\|^2 \de r \bigg]
	\end{equation}
	Moreover for every $q\geq 1$ there is $C_q>0$ (only depending on the dimension $d$ and $q$) such that
	\begin{equation}\label{eqn:GirsanovBM}
		\E \Big[ \psi\!\! \sum_{t\in\unboundedTimeGrid\cap[\tau_1,\tau_2)} \sup_{s\in[t, t+h]} \|W_s -W_t\|^q \Big]
		\leq 2^{q-1} h^{\frac{q}{2}-1} \E \Big[ \psi \big(C_q + \ell^{-q} \big\langle \E_{\circ} [ \psi ] \big\rangle_{\tau_2}^{\frac{q}{2}} \big)\ (\tau_2-\tau_1) \Big]
	\end{equation}
	where $\ell>0$ is the same constant as in the first part of the statement. \close 
\end{lemma}
\begin{proof}
	Consider the density process $ L_t \defined \tfrac{1}{\E^\prob [\psi]} \E_t^\prob [ \psi ] $, $t\geq 0 $.
	Since by assumption $\psi$ is bounded away from zero, there is $\ell>0$ such that $L \geq \ell$. Using elementary properties\footnote{ 
		We use the extend Schwarz's inequality – also known as Kunita-Watanabe inequality, see e.g. \cite[Theorem 2.2.13]{kunita1990stochasticFlows}.
	}
	we compute
	\begin{equation}\label{eqnProofEMMBound}
		\Big| \int_{s}^{t} \frac{1}{L_r} \de\langle\M ,L\rangle_r \Big|^2 
		\leq \Big(\int_{s}^{t}\frac{1}{L_r^2}\de\langle L\rangle_r\Big)\Big(\int_{s}^{t} \|H_r\|^2\de r\Big)
		\leq \ell^{-2} \big\langle\E_{\circ}^\prob [ \psi ] \big\rangle_{t} \int_{s}^{t} \|H_r\|^2\de r
	\end{equation}
	\pagebreak[0]
	\nopagebreak[4]
	for $t\geq s\geq0$. 
	Set $\de\Q \defined L_{\infty} \de \prob$. By Girsanov's theorem 
	$ \M^* \defined \M - \int_0^\cdot \tfrac{1}{L_s} \de \langle\M ,L\rangle_s $
	is a $\Q$- martingale. Use the Kunita-Watanabe identity to express $\M^*$ via the $\Q$-Brownian motion  $W^*$;
	\begin{equation}
		\M^*_t-\M^*_s = \int_{s}^{t} H_r \de W_r - \int_{s}^{t} H_r \tfrac{1}{L_r} \de \langle W ,L\rangle_r = \int_{s}^{t} H_r \de W_r^* \quad\text{for all } t \geq s \geq 0
	\end{equation}
	Using the orthogonality of martingale increments and It\=o's isometry we obtain
	\begin{align}
		\E^\Q \Big[ \sum_{ t\in\unboundedTimeGrid\cap[\tau_1,\tau_2) } |\M^*_{t+h} - \M^*_{t}|^2 \Big] & = \E^\Q \big[ |\M^*_{\tau_2} - \M^*_{\tau_1}|^2 \big] = \E^\Q \Big[ \int^{\tau_2}_{\tau_1} \|H_s\|^2 \de s \Big].
	\end{align}
	Combine this with \eqref{eqnProofEMMBound} (using the monotonicity of the quadratic variation) to estimate
	\begin{equation}
		\E^\Q \Big[\sum_{ t \in\unboundedTimeGrid\cap[\tau_1,\tau_2) } |\M_{t+h}-\M_{t}|^2 \Big] \leq 2\E^\Q \Big[ \big( 1 +\ell^{-2}  \big\langle\E_{\circ}^\prob [ \psi ] \big\rangle_{\tau_2} \big) \int_{\tau_1}^{\tau_2} \|H_r\|^2 \de r \Big]
	\end{equation}
	which under $\prob$ translates to \eqref{eqn:GirsanovMart}, i.e.
	\begin{equation}
		\E^\prob \Big[ \psi \sum_{ t\in\unboundedTimeGrid\cap[\tau_1,\tau_2) } |\M_{t+h} - \M_{t}|^2 \Big] \leq 2 \E^\prob \Big[  \psi \big( 1 +\ell^{-2}  \big\langle\E_{\circ}^\prob [ \psi ] \big\rangle_{\tau_2} \big) \int_{\tau_1}^{\tau_2} \|H_r\|^2 \de r \Big] .
	\end{equation}	
	We remain with establishing the second part: a bound of the random sum of maximal Brownian increments, see \eqref{eqn:GirsanovBM}. Fix $q\geq 2$ and consider 
	$ A_t \defined \sup_{s\in[t, t+h]} \| W_{s} - W_{t} \|^q $ for $t\in\unboundedTimeGrid$. Correspondingly we write $A^*$ for the maximal $W^*$ increments. Observe \eqref{eqnProofEMMBound} for $H=1$ implies
	\begin{equation}
		\sup_{t\in[r,u]}\Big| \int_{r}^{t} \tfrac{1}{L_s} \de\langle W,L \rangle_s \Big|^2 \leq \ell^{-2} \big\langle \E_{\circ}^\prob [ \psi ] \big\rangle_{u}\ (u-r)
	\end{equation}
	Thus with an elementary bound we obtain
	$ A_t \leq 2^{q-1} A_t^* + 2^{q-1} \ell^{-q} \langle \E_{\circ}^\prob [ \psi ] \rangle_{t+h}^{\frac{q}{2}} \, h^{\frac{q}{2}} $.
	Combining those observations yields 
	\begin{equation}
		\tfrac{1}{\E^\prob[ \psi ]} \E^\prob \Big[\psi \sum_{t\in\unboundedTimeGrid\cap[\tau_1,\tau_2)} A_t \Big]
		\leq 2^{q-1} \E^\Q \Big[ \sum_{t\in\unboundedTimeGrid\cap[\tau_1,\tau_2)} A_t^* \Big] + 2^{q-1}\ell^{-q} \E^\Q \Big[ {\big\langle \E_{\circ}^\prob [ \psi ] \big\rangle}_{\tau_2}^{\frac{q}{2}} (\tau_2 - \tau_1) \Big] h^{\frac{q}{2} -1}.
	\end{equation}
	Since (under $\Q$) $A^*$ is an independent and identically distributed sequence we use a version Wald's identity for the discrete-time gird $\unboundedTimeGrid$, see Lemma~\ref{lemma:WaldBound} and obtain 
	\begin{equation}
		\E^\Q \Big[ \sum_{t\in\unboundedTimeGrid\cap[\tau_1,\tau_2)} A_t^* \Big] = \frac{1}{h} \big(\E^\Q [\tau_2] - \E^\Q [\tau_1] \big) \E^\Q [A_0^*] = \frac{1}{h} \E^\Q [ \tau_2 - \tau_1 ] \E^\Q [A_0^*] 
	\end{equation}    
	By definition of $A^*$, Doob's $L^q$-inequality and a well-known bound of the $q$-norm of a $d$-dimensional Gaussian random variable (in terms of the $2$-norm), see e.g. \cite[Corollary 3.2]{ledoux2013probability}, there is $C_q>0$ such that
	$ \E^\Q [A_0^*] = \E^\Q [ \sup_{s\in[0, h]} \| W_{s}^* \|^q ] \leq C_q h^\frac{q}{2} $.
	Combine the bounds and switch to $\prob$ to arrive at the desired bound \eqref{eqn:GirsanovBM}.
\end{proof}
\begin{lemma}[Discrete-Time-Grid Wald]\label{lemma:WaldBound}
	Let $(A_t)_{t\in\unboundedTimeGrid}$ be independent non-negative random variables such that $A_t$ is $\F_{t+h}$-measurable and independent of $\F_t$. Moreover let $\genericStopp$ denote a $\unboundedTimeGrid$-valued 
	stopping time. If $\ \E[\genericStopp]<\infty$ then
	$ \E [ \sum_{t\in\unboundedTimeGrid\cap[0,\genericStopp]} A_t ] \leq \tfrac{1}{h} \E [\genericStopp] \sup_{t\in\unboundedTimeGrid} \E[A_t] $. 
	Particularly whenever the sequence $(A_t)_{t\in\unboundedTimeGrid}$ is identically distributed the bound is an identity.
	\close
\end{lemma}
\begin{proof}
	For $t\in\unboundedTimeGrid$ by definition $\{\genericStopp < t \} \in \F_t$ and this event is independent of $A_t$. Thus
	\begin{equation}
		\E \Big[ \sum_{t\in\unboundedTimeGrid\cap[0,\genericStopp)} A_t \Big] = \sum_{t\in\unboundedTimeGrid} \E [A_t] \prob [\genericStopp > t] 
		\leq \sup_{s\in\unboundedTimeGrid} \E [A_s] \sum_{t\in\unboundedTimeGrid} \prob [\genericStopp > t] = \tfrac{1}{h} \E[\genericStopp] \sup_{s\in\unboundedTimeGrid} \E [A_s]
	\end{equation}
	where in the last step we used that $\genericStopp$ is $\unboundedTimeGrid$-valued. Particularly whenever $(A_t)_{t\in\unboundedTimeGrid}$ are identically distributed we have $\sup_{t\in\unboundedTimeGrid} \E [A_t] = \E [A_0]$ and hence the preceding bound is an identity.
\end{proof}
\subsection*{Moment Estimates involving Quadratic Variation of Closed Martingales}
\begin{lemma}\label{lemma:IntegrabilityQV_ClosedMartingale}
	Let $\genericStopp$ be a stopping time and $V\in L^p$ for some $p>1$. Then
	$ \langle \E_\circ [V] \rangle_{\genericStopp}^{\frac{p}{2}} \in L^1 $. \close
\end{lemma}
\begin{proof}
	We consider the closed martingale $\E_t [V], t\geq 0$. Using the Burkholder-Davis-Gundy, Doob $L^p$ and Jensen inequalities we compute 
	\begin{align}
		\E \Big[ \big\langle \E_\circ [V] \big\rangle_{\genericStopp}^{\frac{p}{2}} \Big] & \leq C_p \E \Big[ \Big(\sup_{t\in [0,\infty)} \E_{t\wedge\genericStopp} [V] \Big)^p \Big] \\
		& \leq C_p \big(\tfrac{p}{p-1}\big)^p \sup_{t\in[0,\infty)} \E \Big[ \big| \E_{t\wedge\genericStopp } [V] \big|^p \Big] 
		\leq C_p \big(\tfrac{p}{p-1}\big)^p\ \E \big[ |V|^p \big] < \infty. \qedhere
	\end{align}
\end{proof}
\begin{lemma}\label{lemma:IntegrabilityQV_PsiGenericStopp}
	Let $p\geq 1$ and $\genericStopp$ be a stopping time. Assume there is $m>0$ such that $\exp(m\genericStopp) \in L^{1}$. Then for any non-negative $V\in L^{4p}$ it holds that
	$ \E [ V ( 1 + \langle \E_\circ [V] \rangle_{\genericStopp}^{p} ) \genericStopp ] < \infty $. \close
\end{lemma}
\begin{proof}
	With the Cauchy-Schwarz inequality we have
	$ \E [ V \langle \E_\circ [V] \rangle_{\genericStopp}^{p} \genericStopp ]^2 
	\leq \E [ \langle \E_\circ [V] \rangle_{\genericStopp}^{2p} ] 
	\E [ V^2 \genericStopp^2 ] $
	and since $V\in L^{4p}$ from Lemma~\ref{lemma:IntegrabilityQV_ClosedMartingale} we know that
	$ \langle \E_\circ [V] \rangle_{\genericStopp}^{2p} \in L^1 $.
	Moreover with Hölder's inequality applied to $q \defined 2p$ and $\Tilde{q} \defined \tfrac{2p}{2p-1}$ we compute
	\begin{align}\label{eqn:ProofIntegrabilityQV_EQ2}
		\E \big[ V^2 \genericStopp^2 \big] \leq \E \big[ V^{4p} \big]^\frac{1}{q} \E \big[ \genericStopp^{2\Tilde{q}} \big]^\frac{1}{\Tilde{q}} < \infty
	\end{align}
	where we use that $\genericStopp$ has a positive exponential moment. 
\end{proof}
\begin{lemma}\label{lemma:IntegrabilityQV_GirsanovGG}
	Let $\genericStopp$ be a finite stopping time and $V\in L^1$ non-negative.    
	Consider the radome variable $ U \defined V \big(1 + \big\langle \E_\circ [V \big] \big\rangle_{\genericStopp} \big) $. If $V\in L^{32} $ then
	$ \E [ U^2 \langle \E_\circ [U ] \rangle_{\genericStopp}^2 ] < \infty $.
	If additionally there is $m>0$ such that $\exp(m\genericStopp)\in L^1$ then 
	$ \E [ U \langle \E_\circ [U ] \rangle_{\genericStopp}^2 \genericStopp ] < \infty $. \close
\end{lemma}
\begin{proof}
	Since $V\in L^{32}$ with the Cauchy-Schwarz inequality and Lemma~\ref{lemma:IntegrabilityQV_ClosedMartingale} we can estimate
	$ \E [ U^8 ]^2 \leq \E [ V^{16} ] 
	\E [ (1 + \langle \E_\circ [V ] \rangle_{\genericStopp} )^{16} ] < \infty $.
	By Lemma~\ref{lemma:IntegrabilityQV_ClosedMartingale} we further have
	$ \langle \E_\circ [U] \rangle_{\genericStopp}^4 \in L^1 $.
	With the Cauchy-Schwarz inequality we conclude
	$ \E [ U^2 \langle \E_\circ [U ] \rangle_{\genericStopp}^2 ]^2 \leq \E [ U^4 ] \E [ \langle \E_\circ [U ] \rangle_{\genericStopp}^4 ] < \infty $.
	%
	The second part of the statement is similar to Lemma~\ref{lemma:IntegrabilityQV_PsiGenericStopp}. We use the Cauchy-Schwarz inequality to estimate
	$ \E [ U \langle \E_\circ [U ] \rangle_{\genericStopp}^2 \genericStopp ]^2 \leq \E [U^2\genericStopp^2] \E [ \langle \E_\circ [U ] \rangle_{\genericStopp}^4 ] $
	where the first factor is finite since $U\in L^4$ and $\genericStopp$ has a positive exponential moment (in \eqref{eqn:ProofIntegrabilityQV_EQ2} choose $q=2$). The second factor is finite since $ \langle \E_\circ [U] \rangle_{\genericStopp}^4 \in L^1 $.
\end{proof}
\end{document}